\documentclass[11pt]{amsart}
\usepackage{hyperref}
\usepackage{amsfonts}
\usepackage{amsmath}
\usepackage{amsthm}
\usepackage{bm}
\usepackage{bbm}
\usepackage{graphicx}
\usepackage{float}
\usepackage{enumerate}
\usepackage{natbib}
\usepackage{amssymb}
\usepackage{amsthm}
\usepackage{latexsym}
\usepackage{color}
\usepackage{mathrsfs}
\usepackage{resizegather}
\usepackage{caption}
\usepackage{subcaption}
\usepackage{float}
\usepackage{pdfsync}
\usepackage{mathtools}
\usepackage{blindtext}
\usepackage[utf8]{inputenc}
\usepackage{makecell}
\usepackage{natbib}

\usepackage[hmarginratio=1:1,top=32mm,columnsep=20pt]{geometry}

\usepackage[T1]{fontenc}
\usepackage{color}	
\usepackage[english]{babel} 

\usepackage{enumitem} 
\setlist[itemize]{noitemsep} 

\usepackage{leftidx}   
\usepackage{mathrsfs}
\usepackage[all]{xy}

\setcitestyle{numbers,square}

\usepackage{filecontents}

\usepackage{fancyhdr}

\setcitestyle{authoryear,open={(},close={)}}

\begin{document}
\allowdisplaybreaks[4]
\newtheorem{theorem}{Theorem}
\newtheorem{lemma}{Lemma}
\newtheorem{pron}{Proposition}

\newtheorem{proposition}{Proposition}[section]

\newtheorem{re}{Remark}
\newtheorem{thm}{Theorem}[section]

\newtheorem{Corol}{Corollary}
\newtheorem{exam}{Example}
\newtheorem{defin}{Definition}
\newtheorem{remark}{Remark}
\newtheorem{property}{Property}
\newtheorem{assumption}{Assumption}
\newcommand{\blue}{\color{blue}}
\newcommand{\red}{\color{red}}
\newcommand{\purple}{\color{purple}}
\newcommand{\la}{\frac{1}{\lambda}}
\newcommand{\sectemul}{\arabic{section}}
\renewcommand{\theequation}{\sectemul.\arabic{equation}}
\renewcommand{\thepron}{\sectemul.\arabic{pron}}
\renewcommand{\thelemma}{\sectemul.\arabic{lemma}}
\renewcommand{\there}{\sectemul.\arabic{re}}
\renewcommand{\thethm}{\sectemul.\arabic{thm}}
\renewcommand{\theCorol}{\sectemul.\arabic{Corol}}
\renewcommand{\theexam}{\sectemul.\arabic{exam}}
\renewcommand{\thedefin}{\sectemul.\arabic{defin}}
\renewcommand{\theremark}{\sectemul.\arabic{remark}}

\def\REF#1{\par\hangindent\parindent\indent\llap{#1\enspace}\ignorespaces}
\def\lo{\left}
\def\ro{\right}
\def\be{\begin{equation}}
\def\ee{\end{equation}}
\def\beq{\begin{eqnarray*}}
\def\eeq{\end{eqnarray*}}
\def\bea{\begin{eqnarray}}
\def\eea{\end{eqnarray}}
\def\o{\overline}
\newcommand{\bH}{\mathbf{H}}
\newcommand{\bg}{\mathbf{g}}
\newcommand{\ba}{\mathbf{a}}
\newcommand{\bb}{\mathbf{b}}
\newcommand{\bT}{\mathbf{T}}
\newcommand{\bh}{\mathbf{h}}
\newcommand{\bY}{\mathbf{Y}}
\newcommand{\cid}{\stackrel{d}{\to}}
\newcommand{\cip}{\stackrel{P}{\to}}

\newcommand{\reals}{{\mathbb R}}
\newcommand{\bbr}{\reals}
\newcommand{\eid}{\overset{d}{=}}

\newcommand{\bx}{{\bf x}}
\newcommand{\by}{{\bf y}}
\newcommand{\bz}{{\bf z}}
\newcommand{\bu}{{\bf u}}
\newcommand{\bv}{{\bf v}}
\newcommand{\bw}{{\bf w}}
\newcommand{\bt}{{\bf t}}

\newcommand{\BX}{{\bf X}}
\newcommand{\BY}{{\bf Y}}
\newcommand{\BW}{{\bf W}}
\newcommand{\BT}{{\bf T}}
\newcommand{\BZ}{{\bf Z}}
\newcommand{\BG}{{\bf G}}

\newcommand{\BPsi}{\text{\boldmath $\Psi$}}
\newcommand{\bxi}{\text{\boldmath $\xi$}}
\newcommand{\bgamma}{\text{\boldmath $\gamma$}}

\newcommand{\cov}{\text{Cov}}
\newcommand{\Cov}{\text{Cov}}

\newcommand{\one}{{\bf 1}}
\newcommand{\bbz}{\mathbb Z}
\newcommand{\bbc}{\mathbb C}
\newcommand{\PP}{\mathbb P}

\newcommand{\calF}{\mathcal F}
\newcommand{\calH}{\mathcal H}

\newcommand{\EE}{\mathbb E}

\newcommand{\vep}{\varepsilon}

\title[Preferential Attachment with Deletion]{How Robust is Scale-Free Structure? Phase Transitions in Preferential Attachment Model under Attack}

\author{Xiaoyun Gong}
\address{Center for Applied Mathematics\\
Cornell University}
\email{xg332@cornell.edu}

\author{Gennady Samorodnitsky$^*$}
\address{School of Operations Research and Information Engineering\\
Cornell University}
\email{gs18@cornell.edu}

\numberwithin{equation}{section}
\thanks{ $^*$The corresponding author. Research  partially
  supported by AFOSR grant FA9550-26-1-B146 at Cornell University.}



 \begin{abstract}
Will an attack on a network affect its structure in a major way? In
this work we study this question by investigating a dynamic preferential
attachment-type graph in which an attacker interferes with the growth
mechanism. We show that our model of an attack exhibits a phase
transition. If the skill of the attacker is low, the network
preserves its scale-free structure inherited from the underlying
preferential attachment structure (even though the tail of the degree
distribution becomes lighter due to the attack). When the skill of the
attacker increases and crosses a critical boundary, the scale-free
property is destroyed. The skill of the attacker is  expressed as the
probability that, at a given step, the attacker succeeds in preventing
adding an edge to the network and, instead, causes a deletion (in a
preferential manner) of an edge. In all cases we describe the tail 
behaviour  of the degree distribution. 
\end{abstract}

\maketitle

\section{Introduction} \label {sec:prelim}
\setcounter{thm}{0}\setcounter{Corol}{0}\setcounter{lemma}{0}\setcounter{pron}{0}\setcounter{equation}{0}
\setcounter{remark}{0}\setcounter{exam}{0}\setcounter{property}{0}\setcounter{defin}{0}
Networks are all around us, and these networks are under attack. The
terms such as denial of service, hacking and ransomware have,
unfortunately, firmly entered our lives. It is, therefore, important
to understand how resilient are networks with different structure
against an attack. In this work we take a step towards that
understanding by considering a network modeled as a random graph, with a
particular structure, that of a preferential attachment type, and
investigate the resilience of the resulting network. Our choice of the
model is based on its popularity, mathematical tractability and belief
that it possesses  certain robustness properties that other types of
networks may not have, going back to \cite{albert2000error}. 

Robustness of networks modeled as graphs has been a major topic of
recent interest, with a significant survey of literature on graph vulnerability
depending on the network topology, type of the attack and robustness
measures given by \cite{freitas:yang:kumar:tong:chau:2023}. Resilience 
of preferential attachment graphs have been attracted much attention,
both heuristically and computationally, and rigorously, starting with
\cite{bollobas2004robustness} who investigated the effect of removing
a positive fraction of all nodes on the connectivity of the
graph. Later, \cite{eckhoff:morters:2014} showed that a removal of an
arbitrarily small fraction of the ``oldest'' nodes causes (among other
major changes) a complete disappearance of the power-tail distribution
of the node degree (also known as the scale-free property), pointing
to fragility of the network to a completely informed attack on the
final state of the system.

Preferential attachment models describe a growing sequence of graphs,
so a different type of attack in this model would attempt to interfere
with the growing mechanism of the system (as opposed to damaging a
part of the network after a (large) number of steps.) Once again,
heuristic and computational analysis came first. A model with a
fractional number $0<r<1$ of nodes deleted each time a new node is
added was considered In 
\cite{moore:ghoshal:newman:2006}, and the scale-free property was
found to always remain, with the power exponent in the degree
distribution an increasing function of $r$. A different type of
mechanism interference was considered in
\cite{deijfen:lindholm:2009}. Here, at each step, and attacker succeeds, 
with a certain probability, say, $1-\pi$, to prevent the network from adding an edge
and, instead, causes and edge to be deleted (with probability $\pi$
the attack fails and the networks successfully growth at that step). This paper finds a phase
transition: if the probability  of a successful attack $1-\pi<1/3$,
then the scale-free property is preserved, while if $1-\pi>1/3$,
then the scale-free property disappears (in fact, the limiting degree
distribution is found to have an exponentially fast decaying tail in
the range $1/3<1-\pi<1/2$).

The model  we are considering is, in a sense, a generalization of the
model in \cite{deijfen:lindholm:2009}, though our graphs are directed
and we eliminate a feature
in the latter model that appears to be less informative with respect
to the system robustness. Our analysis is completely rigorous, and we
confirm the the finding of a phase transition announced in
\cite{deijfen:lindholm:2009}. In fact, we find a critical curve in the
space of two parameters, one of which is the probability an attack
succeeds at a given step, and the second is the parameter governing
the strength of preferential attachment. On one side of this critical
curve the scale-free property is preserved, while on its other side
the limiting degree distribution has an exponentially fast decaying
tail (with an intermediate type of behavior in the critical regime
itself). 

This paper is organized as follows. In Section \ref{sec:2models} we
describe formally the model and discuss its features. The main results
of the paper are presented in Section \ref{sec:general}. We prove
convergence of the empirical distribution of the (in)-degree of the
vertices in the dynamic graph to the limiting distribution in Section
\ref{sec:convergence}. 

Finally, the Appendix collects a number of calculations and results
used throughout the paper.

In the sequel, unless otherwise specified, $C$ will denote a finite
positive constant whose exact value is immaterial, and which may
change from one appearance to the next.

\section{The model} \label{sec:2models}
\setcounter{thm}{0}\setcounter{Corol}{0}\setcounter{lemma}{0}\setcounter{pron}{0}\setcounter{equation}{0}
\setcounter{remark}{0}\setcounter{exam}{0}\setcounter{property}{0}\setcounter{defin}{0}

The dynamics of preferential attachment graphs incorporates the idea
of ``rich gets richer''. Such models were introduced in
\cite{barabasi:albert:1999}. This model turns out to produce
``scale-free graphs'', i.e. a sequence of graphs whose limiting degree
distribution has a 
tail with a power-type of decay, as opposed to exponential, or
faster-than-exponential, type of decay for many other random graph
models. The preferential attachment mechanism, where the
``newcomers'' are more  
likely to attach to already well-connected vertices, is seen as
supported by empirical studies such as
\cite{newman:2001b,jeong:neda:barabasi:2003}. It has been used, in
particular,  to
understand the dynamics of networks in biology, such as protein network
and bacterial surface colonization
\cite{grinberg:orevi:kashtan:2019,eisenberg:levanon:2003}. 

Since its introduction, numerous variants of the preferential
attachment model have been studied. This included varying number of
edges added to the graph at each step (see
e.g. \cite{bollobas:riordan:spencer:tusnady:2001}), including adding a random
number of edges as in
\cite{deijfen:esker:hofstad:hooghiemstra:2009}. Directed preferential
attachment graphs were also introduced
(e.g. \cite{bollobas:borgs:chayes:riordan:2003}),  and the
multivariate 
power-tail behaviour of the joint distribution of the in-degree and
out-degree was established in
\cite{samorodnitsky:resnick:towsley:davis:willis:2016}. One could also
attach an intrinsic fitness value to each node in the network; the
product of the degree of a node and its fitness then determines how
``attractive'' this node is for a future connections; see
\cite{bianconi:barabasi:2001}. One could establish 
central limit theorems for the degree counts; see e.g. \cite{resnick2016asymptotic}.

In order to see clearly the effect of an attacker interfering with the
growth mechanism of the system,  we choose a relatively simple
model of preferential attachment. The model has two parameters,
$1/2<\pi<1$ and $\delta>0$. 

We construct a  sequence of directed 
graphs $\{G(n), \, n=0,1,2,\dots\}$, with $V_n$ and $E_n$ denoting
the vertex set and the edge set of $G(n)$. Their corresponding
cardinalities are denoted by $v_n = |V_n|$ and $e_n = |E_n|$. For a
vertex $u\in V_n$ we denote by $d_n(u)$ the in-degree of $u$ in
$G(n)$.

The process is initialized   with a graph $G(0)$ consisting of two
vertices connected by a single directed edge from one vertex to the
other.  For $n \geq 0$, given the graph $G(n)$, the graph $G(n+1)$ is
constructed according to the following update rule:

\begin{itemize}
    \item with probability $\pi$ if $e_n>0$, or with probability 1 if
      $e_n=0$, a new vertex $w\notin V_n$, with a 
      directed, outgoing,       edge $(w,u)$ is introduced.  The
      second node $u$ is  either an existing vertex
      $u\in V_n$ or $u=w$, (resulting in a loop).
    The probabilities of the different choices are 
    \begin{equation} \label{e:connection.prob}
    \mathbb{P}(w \text{ connects to } u) = \begin{cases}
    \dfrac{d_n(u)+\delta}{e_n+\delta v_n +\delta} & \text{for} \ u \in
    V_n,\\
    \dfrac{\delta}{e_n+\delta v_n +\delta} & \text{for} \  u = w.
    \end{cases}
    \end{equation}
  That is, $V_{n+1} = V_n\cup \{w\}$ amd $E_{n+1}= E_n\cup
  \{(w,u)\}$. 
    \item if $e_n>0$, then with probability $1-\pi$ an existing edge
      is chosen uniformly at random among all edges in $E_n$ 
and removed from the set of edges, resulting in the new set of
edges $E_{n+1}$. In this case we set $V_{n+1}=V_n$.  
\end{itemize}

\begin{remark}
{\rm   
We have chosen a directed graph dynamics to simplify the
analysis. Analogous calculations for an indirected version of the
model would be more involved and, we believe, the overall picture
is likely to be the same.

The parameter $\pi$ is the probability that the attempt of the
attacker to interfere with preferential growth  at a given step fails. A
comparison with a simple random walk for the number of edges in the
graph shows that, if  $\pi \leq 1/2$, then 
$\mathbb{P}(e_n=0 \text{ i.o.})=1$, so eventually most vertices will
have degree 0. In the range $\pi\in (1/2,1]$  the
limiting degree distribution is non-trivial.

The parameter $\delta>0$ moderates the strength of preferential
attachment: the larger is $\delta$, the weaker the effect of
preferential attachment is.
}
\end{remark}

\begin{remark} \label{rk:edge.c}
{\rm 
We emphasize that, when the attacker succeeds in disrupting, at a given step,
preferential growth, an edge is deleted in preferential manner as
well: choosing an edge for deletion uniformly at random is equivalent
to selecting a vertex with probability proportional to its in-degree
and then deleting an edge entering this vertex uniformly at random
among all such edges. One can, therefore, view our model model as
including an informed attacker, capable of attacking the key parts of
the network. 

}
\end{remark} 
 
 In our model described above, $v_n$ and $e_n$, the numbers of
 vertices and edges in $G(n)$ are random, and so is the denominator in
 \eqref{e:connection.prob}, a feature shared by only few
 of the previously studied models of a preferential attachment
 type. This introduces an extra technical difficulty that we overcome
 in the sequel.

\section{Main results} \label{sec:general}
\setcounter{thm}{0}\setcounter{Corol}{0}\setcounter{lemma}{0}\setcounter{pron}{0}\setcounter{equation}{0}
\setcounter{remark}{0}\setcounter{exam}{0}\setcounter{property}{0}\setcounter{defin}{0}


For $n\geq 0$ we denote by $N_k(n)$ the number of vertices with
in-degree $k$ at time $n$, i.e. 
$$
N_k(n) = \sum_{v\in V_n} \mathbbm{1}_{\{d_n(v)=k\}}, \ k=0,1,2,\ldots,
$$
so that
$$
p_k(n) = N_k(n)/v_n, \ k=0,1,2,\ldots,
$$
is the p.m.f. of the empirical distribution of the in-degrees of the
vertices of the graph $G(n)$. As in the many other models of the
preferential attachment type, it turns out that, with probability 1,
this empirical distribution converges weakly to a non-random limit. 
Once we establish this fact, we will turn to our main goal, that of
understanding whether the tail of the limiting distribution has a
power-like decay, or an exponential-like decay. In other words, does
the network still has the scale-free property it would have if not for
the interfering attacker? It turns out that both situations are
possible, depending on the skill of the attacker, as measured by the
probability of successful interference $1-\pi$, and on the moderating
effect on the strength of the preferential attachment induced by the
parameter $\delta$. That is, a phase transition occurs at the boundary
between these two regimes, and we will describe that  boundary.


We start with establishing existence of a non-random limiting degree
distribution. The proof of the next theorem is given in Section
\ref{sec:convergence}.  In the statement of the theorem we will use
the following notation. 
\begin{align} \label{e:all.coeff}
\alpha_0 &= \frac{\pi}{2\pi-1 + \delta\pi},
& \beta_0 &= \frac{\delta\pi}{2\pi-1 + \delta\pi}, \\[0.4em]\notag 
\alpha_1 &= -\frac{\pi}{2\pi-1 + \delta\pi} - \frac{1-\pi}{2\pi-1},
& \beta_1 &= -\frac{\delta\pi}{2\pi-1 + \delta\pi} - 1,
  \\[0.4em] \notag 
\alpha_2 &= \frac{1-\pi}{2\pi-1},
& \beta_2 &= 0.
\end{align}

\begin{thm}\label{thm:convergence+concentration}
    For each fixed $k=0,1,2,\ldots$, the limit
    $        p_k = \lim_{n\to\infty} \mathbb{E} [p_k(n)]    $
    exists, and $(p_k, \, k=0,1,2,\ldots)$ is a p.m.f. on the
    nonnegative integers. In fact, there are positive $\kappa, C$
    such that 
    \begin{equation} \label{e:total.var}
    \sum_{k=0}^{\infty}\left|p_k-\mathbb{E} [p_k(n)]\right| \leq Cn^{-\kappa}
    \end{equation}
    for all $n$. Moreover, for any $\epsilon > 0$, there exist $C>0$
    and  $N$ such that for all $n \ge N$ and all $k$,
    \begin{equation} \label{e:nonrand.conv}
      \mathbb{P}\left( \left|p_k(n)-p_k\right| > \epsilon\right) \leq
      Ce^{-C^{-1}\min(\epsilon,\epsilon^2)n/(\log n)^{2}}. 
    \end{equation}
The limiting probabilities $(p_k, \, k=0,1,2,\ldots)$ satisfy the
recursion 
  \begin{equation} \label{eqn:rr}
    [\alpha_2(k+1)+\beta_2]p_{k+1}+[\alpha_1k+\beta_1]p_{k}+[\alpha_0 (k-1)+\beta_0]p_{k-1} = -\mathbbm{1}_{k=0},
\end{equation}
$k=0,1,\ldots$, with $p_{-1}=0$. 
\end{thm}


\begin{remark}
  {\rm 
Observe that once \eqref{e:total.var} has been established, the fact
that $(p_k, \, k=0,1,2,\ldots)$ is a true probability measure
follows. Furthermore, it will follow from \eqref{e:nonrand.conv} and
the first Borel-Cantelli lemma that $p_k(n)\to p_k$ a.s for every $k$
and, hence, on an event of probability 1, the sequence $(p_k(n), \,
k=0,1,2,\ldots)$ of the
empirical in-degree distributions of the vertices in $G(n), \, n\geq
0$ converges weakly to the probability measure given by  $(p_k, \,
k=0,1,2,\ldots)$. 
  }
\end{remark}
\begin{remark}
  {\rm
While the statements in Theorem \ref{thm:convergence+concentration}
are analogous to the corresponding statements for the models without
an attack that have appeared in a number of prior publications, the
arguments required to establish these results when the dynamics
incorporates an attacker, are significantly more involved. In
particular, to establish the concentration result
\eqref{e:nonrand.conv} we apply a martingale approach
used in  \cite{bollobas:riordan:spencer:tusnady:2001}. In our
situation, however, the martingale approach requires a non-trivial
coupling construction unnecessary in the previous cases. Our approach
will also yield the following concentration bound for the in-degree
counts, For some $C>0$, as $n\to\infty$, 
\begin{equation}\label{prop:concentration_count}
    \mathbb{P}\left(\max_k \left|N_k(n)-\mathbb{E}[N_k(n)]\right| \geq
      C\sqrt{n}(\log n)^{3/2}\right) = o(1). 
 \end{equation}
  }
\end{remark} 

We now turn to the main result of this paper, describing the tail
behaviour of the limiting probabilities $(p_k)$ and the phase
transition observed when the attack is severe enough to destroy the
scale-free property.  Theorem
\ref{thm:tails} below shows that the critical boundary in the space of
the model parameters $(\pi,\delta)$ is that of $\alpha_0=\alpha_2$,
where  $\alpha_0$ and $\alpha_2$ are as in \eqref{e:all.coeff}.
Specifically, we will see that, as $k$ increases, 
\begin{itemize}
    \item if $\alpha_0 > \alpha_2$, then $p_k$ decays in a
      power-like manner, so the scale-free property remains;
   \item if $\alpha_2>\alpha_0$, then $p_k$ decays in an
      exponential-like manner, so the scale-free property is destroyed.
    \end{itemize}
On the critical boundary $\alpha_0=\alpha_2$ itself an intermediate type of
decay of $p_k$ is observed. 

The following notation will be important in the sequel.  If
$\alpha_0\not=\alpha_2$ we set 
\begin{equation} \label{e:AB}
B = \frac{1}{\alpha_0 - \alpha_2}, \ \ 
A = -(\delta + B), 
\end{equation}
and we use the notation $\pi_*=3.14...$ to distinguish it from a model
parameter $\pi$. 
\begin{thm}\label{thm:tails}
As $k\to\infty$, 
\begin{align} \label{e:pk.asymp}
    p_k \sim \begin{cases}
        c_{\pi,\delta}(\alpha_0/\alpha_2)^{k} \cdot k^{-(A+1)} &
        \text{if } \ \alpha_0<\alpha_2,\\
        c_{\pi,\delta} e^{-2\alpha_0^{-1/2}k^{1/2}}\cdot k^{\delta/2-3/4} & \text{if
        } \ \alpha_0=\alpha_2,\\
        c_{\pi,\delta} k^{-(B+1)} & \text{if } \ \alpha_0>\alpha_2, 
    \end{cases}
\end{align}
with
\begin{align*}
  c_{\pi,\delta} = \begin{cases}
    (\alpha_2-\alpha_0)^{-(A+1)}\alpha_2^A\Gamma(-B)/\Gamma(\delta) & \text{if
} \ \alpha_0<\alpha_2, \\
e^{1/(2\alpha_0)}\alpha_0^{-(\delta/2+1/4)}\sqrt{\pi_*}/\Gamma(\delta) &  \text{if
        } \ \alpha_0=\alpha_2,\\
(\alpha_0-\alpha_2)^{-(B+1)}\alpha_0^B\Gamma(-A)/\Gamma(\delta) & \text{if
} \ \alpha_0>\alpha_2. 
\end{cases}
\end{align*}

\end{thm}

\begin{remark}
  {\rm
    Figure~\ref{fig:partition_delta_pi} presents the phase transition
    described in Theorem \ref{thm:tails} in the space of the model
    parameters $(\pi,\delta)$. We see that the scale-free property is
    preserved when $\pi$ is large and $\delta$ is small. More
    specifically, for any $\delta>0$ there is a critical value
    $\pi_*(\delta)\in (1/2,1)$ such that the scale-free property is
    preserved in the range $\pi \in (\pi_*(\delta),1)$ and disappears
    in the range $\pi \in (1/2,\pi_*(\delta))$. On the other hand,
    for any $\pi\in (1/2,1)$ there is a critical value $\delta_*(\pi)$
    such that the scale-free property is
    preserved in the range $0<\delta<\delta_*(\pi)$ and disappears in
    the range $\delta>\delta_*(\pi)$. If the attacker succeeds only
    rarely ($\pi$ close to 1), then the scale-free property is
    preserved almost no matter what $\delta$ is, and when of the
    preferential attachment is very strong ($\delta$ close to 0), then
    the scale-free property is 
    preserved almost no matter what $\pi$ is.

  When $\pi=1$ (no attacker present), we have
  $\alpha_0=1/(1+\delta)>0=\alpha_2$, and $B=1+\delta$. In this case
  the statement of Theorem \ref{thm:tails} coincides with the
  classical result on the tail of the tail distribution in the
  preferential attachment model; see Section 8 in
  \cite{vanderHofstad:2016} (with $m=1$ and $\delta$ replaced by
  $\delta+1$ to account for the constant out-degree of 1 not used in
  our directed model).

\begin{figure}[H]
  \centering
  \includegraphics[scale=0.5]{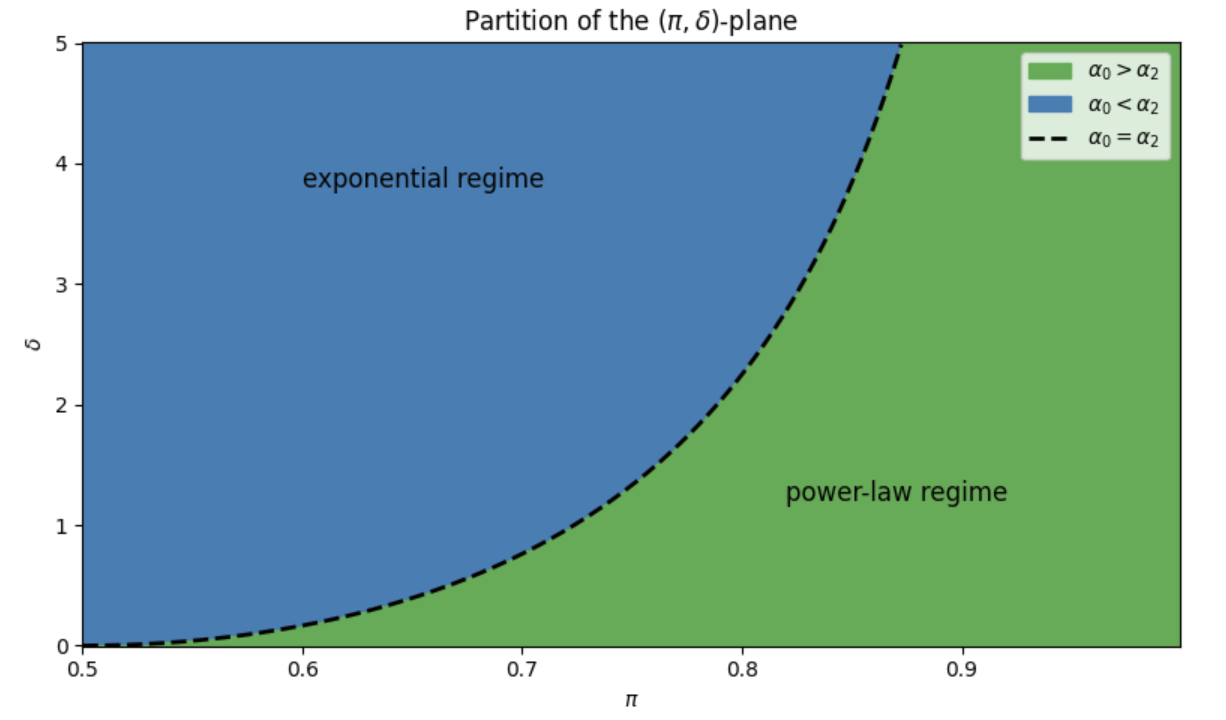}
  \caption{The phase transition between presence and absence of the
    scale-free property in the $(\pi,\delta)$-plane}
  \label{fig:partition_delta_pi}
\end{figure}

}
\end{remark}

\section{Convergence and Concentration} \label{sec:convergence}
\setcounter{thm}{0}\setcounter{Corol}{0}\setcounter{lemma}{0}\setcounter{pron}{0}\setcounter{equation}{0}
\setcounter{remark}{0}\setcounter{exam}{0}\setcounter{property}{0}\setcounter{defin}{0}

In this section, we prove the convergence and concentration results in
Theorem~\ref{thm:convergence+concentration}.  We start with deriving a
recursion relating the expected vertex count
$\mathbb{E}[N_k(n+1)]$ to certain expected vertex counts on the
previous step. We then couple the resulting recursive relation with a
different recursion, which is satisfied by the claimed limiting probabilities.
By conditioning on the graph $G(n)$, we obtain
\begin{equation*}
\begin{aligned}
    \mathbb{E}\bigl[N_k(n+1)-N_k(n) \mid G(n)\bigr] &= \mathbbm{1}_{\{ e_n > 0 \}} \cdot\bigl(\pi C^{(1)}_{k,n} + (1-\pi) C^{(2)}_{k,n}\bigr) + \mathbbm{1}_{\{ e_n = 0 \}} \cdot  C^{(1)}_{k,n},
 \end{aligned}
\end{equation*}
where 
$$C^{(1)}_{k,n} = \frac{(k-1+\delta)N_{k-1}(n)}{e_n+\delta v_n+\delta}-\frac{(k+\delta)N_{k}(n)}{e_n+\delta v_n+\delta} + \frac{e_n+\delta v_n}{e_n+\delta v_n+\delta} \mathbbm{1}_{k=0} + \frac{\delta}{e_n+\delta v_n+\delta}\mathbbm{1}_{k=1},$$
$$C^{(2)}_{k,n} = \frac{(k+1)N_{k+1}(n)}{e_n} -\frac{kN_{k}(n)}{e_n}.$$
Taking the expectation on both sides and using the bounds obtained
in Section \ref{sec:approx} of the appendix that allow us to replace the random
quantities $e_n$ and $v_n$ by their linear approximations, we obtain 
\begin{align} 
 \notag \mathbb{E}[N_k(n+1)] - \mathbb{E}[N_k(n)] &=\mathbb{E}\bigl[
\bigl( \pi \cdot  \mathbbm{1}_{\{ e_n > 0 \}}
      + \mathbbm{1}_{\{ e_n = 0 \}}\bigr)   C^{(1)}_{k,n} \bigr] + (1-\pi) \mathbb{E}\bigl[\mathbbm{1}_{\{ e_n > 0 \}}C^{(2)}_{k,n}\bigr] \\
\label{eqn:rr_approximation}    &= \pi \cdot \left[\frac{(k-1+\delta)\mathbb{E}[N_{k-1}(n)]}{(2\pi-1) n+\delta \pi n}-\frac{(k+\delta)\mathbb{E}[N_{k}(n)]}{(2\pi-1) n+\delta \pi n} + \mathbbm{1}_{k=0} \right]\\
 \notag   &+ (1-\pi) \cdot \left[
            \frac{(k+1)\mathbb{E}[N_{k+1}(n)]}{(2\pi-1) n}
            -\frac{k\mathbb{E}[N_{k}(n)]}{(2\pi-1) n} \right] +  R_{n,k},
\end{align}
where we set $\mathbb{E}[N_{-1}(n)]=0$ and 
$R_{n,k}$ is the error resulting from the approximations. Recall
that
\begin{equation} \label{e:R.L1}
\sum_{k=0}^\infty |R_{n,k}|= O(n^{-\epsilon}) \ \text{for any} \
0<\epsilon<1/2\,.
\end{equation}

Proposition \ref{pr:heart} below 
shows that the recursion \eqref{eqn:rr} has a solution $(p_k, \,
k=0,1,2,\ldots)$ that satisfies 
\begin{equation} \label{e:bound.pk}
|p_k|\leq Ck^{-1-\gamma}, \ k=1,2,\ldots
\end{equation}
for some $\gamma, C>0$. We fix one such solution and consider ``the
error'' vector
\begin{equation} \label{e:en}
e_n = \bigl( e_{n,0}, e_{n,1}, \ldots\bigr) \ \ \text{with} \ \ e_{n,k} =
\mathbb{E}[N_k(n)]/n-\pi p_k, \ k=0,1,2,\ldots.
\end{equation}
We will show that there are positive $\kappa, C$
    such that 
\begin{equation} \label{e:bound.en}
\|e_n\|_1 \leq Cn^{-\kappa}.
    \end{equation}
This, of course, will imply that $\mathbb{E} [p_k(n) ] \to p_k$ for each
$k$ as well as establish \eqref{e:total.var} (with a different $C$).

We start by noticing that the recursion \eqref{eqn:rr} can be rewritten in the form    
$$
    p_k = \pi \cdot \left[\frac{(k-1+\delta)p_{k-1}}{(2\pi-1) +\delta \pi }-\frac{(k+\delta)p_k}{(2\pi-1) +\delta \pi }\right]+ (1-\pi) \cdot \left[ \frac{(k+1)p_{k+1}}{2\pi-1 } -\frac{kp_k}{2\pi-1} \right] + \mathbbm{1}_{k=0},
$$
with $p_{-1}=0$. Comparing this with the recursion
\eqref{eqn:rr_approximation} one sees that the vectors $(e_n)$ satisfy
\begin{align}\label{eqn:ee}
    e_{n+1} &= \left(I+\frac{1}{n+1}(L-I)\right)e_n + r_n
\end{align}
where $I:\, \bbr^\infty\to\bbr^\infty$ is the identity map, 
$L$ is an infinite matrix with rows and columns numbered
$0,1,2,\ldots$, with 3 non-zero entries in each row (except the zeroth
row which has 2 non-zero entries), such that for each infinite vector
$y=(y_0,y_1,y_2,\ldots)$ 
$$
(Ly)_k = \frac{\pi}{(2\pi-1)+\delta \pi}
\bigl((k-1+\delta)y_{k-1}-(k+\delta)y_k\bigr) +
\frac{1-\pi}{2\pi-1}\bigl((k+1)y_{k+1}-ky_k\bigr), \ k=0,1,2,\ldots
$$
and the term $y_{-1}$ absent when $k=0$. Furthermore, $r_n$ is an
infinite vector with entries $r_{n,k} = R_{n,k}/(n+1)$ for all $k$. It
follows from \eqref{e:R.L1} that
\begin{equation} \label{e:R.L2}
\|r_n\|_1= O\bigl(n^{-(1+\epsilon)}\bigr) \ \text{for any} \
0<\epsilon<1/2\,.
\end{equation}

Observe that $L$ is an infinite tridiagonal matrix. The entries on its main diagonal 
are negative, while the entries on the two adjacent diagonals are
positive. Moreover, all columns of $L$ have zero column sum, and if
$L$ is truncated to a finite square matrix consisting of the first
rows and the first columns of $L$, then the column sums of the
truncated matrix are nonpositive. 

Let $P_m:\, \bbr^\infty\to\bbr^\infty$ be the projection of a vector in $\bbr^\infty$ 
onto its first $m$ coordinates. For a positive integer $K(n)$ to be
specified momentarily we decompose the vector $e_n$ in
\eqref{e:en} as 

\begin{equation} \label{e:decompose.e}
u_n = P_{K(n)}e_n,\qquad v_n = e_n-u_n,
\end{equation}
so that $e_n=u_n+v_n$. We choose $K(n)$ as follows. Let  
$$
c_0 = \left(\frac{\pi}{(2\pi-1)+\delta
    \pi}+\frac{1-\pi}{2\pi-1} +1\right)^{-1}
$$
and choose $c'\in [5c_0/8,c_0]$. Let $n_0$ be a large integer.  For a carefully
chosen sequence $c_i\in [5c_0/8,c_0], \, i=1,2,\ldots$ we construct an
integer-valued function $K$ on $\{ n_0,n_0+1,\ldots\}$  by setting
$K(n_0) = \lfloor c'n_0\rfloor$ and 
\begin{equation} \label{e:Kn}
K(n) = \begin{cases}
    K(n_0) + \lfloor c_1 (n-n_0)\rfloor & n\in[n_0,3n_0]\\
    K(3n_0)+ \lfloor c_2 (n-3n_0)\rfloor & n\in [3n_0,9n_0]\\
    \dots\\
    K(3^{k-1}n_0)+\lfloor c_k (n-3^{k-1}n_0)\rfloor & n\in [3^{k-1}n_0,3^kn_0]\\
    \dots
  \end{cases}
  .
\end{equation}
For large enough $n_0$ this ensures that $(3c_0/5)(3^kn_0)\leq
K(3^kn_0)\leq c_0(3^kn_0)$ for all $k\geq 0$ and 
$c_0n/2\leq K(n)\leq c_0n$ for all $n\geq n_0$. Since $c_i<1$ for all
$i$, $K(n)\leq
K(n+1)\leq K(n)+1$ for all $n\geq n_0$. 

We claim that we can choose the sequence $(c_i)$ in such a way that
for some $C>0$, for every  $k=1,2,\dots$, the following condition is
satisfied: 
\begin{align} \label{e:good.K}
&\Big|n\in[3^{k-1}n_0,3^kn_0]: \mathbb{E}[N_{K(n+1)}(n)]>n^{-0.1} \\
\notag &\hskip 1.1in \text{ or } \mathbb{E}[N_{K(n+1)+1}(n)]>n^{-0.1} \}\Big|\leq C\cdot
(3^{k-1}n_0)^{0.2}.
\end{align}
We will refer to $n$ such that one of the conditions in the left-hand
side of \eqref{e:good.K} is satisfied as ``bad''; otherwise we will refer to $n$ as
``good''. 

Assuming, for a moment, that such a sequence $(c_i)$ exists, we proceed
to bounds the norms of the two terms in the decomposition of $e_n$ in
\eqref{e:decompose.e}. Starting with the term $u_n$, we have for
$n\geq n_0$, 
\begin{equation}\label{eqn:u_n}
   \begin{aligned}
    u_{n+1} &= P_{K(n+1)} e_{n+1} = P_{K(n+1)}\left(\left(I+\frac{1}{n+1}(L-I)\right)e_n\right) + P_{K(n+1)}r_n
\end{aligned} 
\end{equation}
$$
    = P_{K(n+1)}\left(\left(I+\frac{1}{n+1}(L-I)\right)u_n\right) + P_{K(n+1)}\left(\left(I+\frac{1}{n+1}(L-I)\right)v_n\right) + P_{K(n+1)}r_n.
$$
By \eqref{e:R.L2}, $\|P_{K(n+1)}r_n\|_1 =
O(n^{-1-\epsilon})$. Further,
$$
\begin{aligned}
&\left\| P_{K(n+1)}\left(\left(I+\frac{1}{n+1}(L-I)\right)v_n\right) 
\right\|_1 \leq |(e_n)_{K(n+1)}| \\
& \hskip 0.3in +\frac{1}{n+1}
\cdot \left( \frac{\pi(K(n+1)+\delta)}{(2\pi-1)+\delta
    \pi}+\frac{(1-\pi)K(n+1)}{2\pi-1}+1 \right)|(e_n)_{K(n+1)}|\\
& \hskip 0.3in + \frac{1}{n+1} \cdot
\frac{(1-\pi)(K(n+1)+1)}{2\pi-1}\bigl( |(e_n)_{K(n+1)}|+|(e_n)_{K(n+1)+1}|\bigr), 
\end{aligned}
$$
where we used the fact that $\|P_{K(n+1)} v_n\|_1 \leq
|(e_n)_{K(n+1)}|$, and if $(w_n)_k = (v_n)_{k+1}$ for all $k$, then $\|P_{K(n+1)} w_n\|_1 \leq
|(e_n)_{K(n+1)}|+|(e_n)_{K(n+1)+1}|$. 
Notice that for every ``good'' $n$, 
$$
|(e_n)_{K(n+1)}|\leq \mathbb{E}[N_{K(n+1)}(n)]/n+\pi p_{K(n+1)}\leq
n^{-1.1}+CK(n+1)^{-1-\gamma}\leq Cn^{-1-\min(0.1,\gamma)}, 
$$
and similarly
$|(e_n)_{K(n+1)+1}|\leq Cn^{-1-\min(0.1,\gamma)}$.  Therefore for the ``good''  $n$, 
\begin{align*}
\left\| P_{K(n+1)}\left(\left(I+\frac{1}{n+1}(L-I)\right)v_n\right)
\right\|_1
 \leq Cn^{-1-\min(0.1,\gamma)}. 
\end{align*}
On the other hand, for the ``bad'' $n$ we use the trivial bound
$$
N_{K(n+1)}(n) \leq e_n/K(n+1) \leq (n+1)/[c_0(n+1)/2] =2/c_0
$$
to obtain a worse bound
\begin{align*}
\left\| P_{K(n+1)}\left(\left(I+\frac{1}{n+1}(L-I)\right)v_n\right)
\right\|_1
 \leq Cn^{-1}. 
\end{align*}

Next, notice that by the choice of $c_0$ and the fact that $K(n)\leq
c_0n$ for all $n\geq n_0$,  the matrix $I+\frac{1}{n+1}(L-I)$
restricted to its first $K(n)+1$ rows and columns is nonnegative, as
long as $n_0$ is large enough.  Therefore,  
\begin{align*}
    \left\|P_{K(n+1)}\left(\left(I+\frac{1}{n+1}(L-I)\right)u_n\right)\right\|_1
  &\leq \max_{0\leq j\leq K(n)}\left\|\left(I+\frac{1}{n+1}(L-I)\right)_{\cdot j} \right\|_1\|u_n\|_1\\
    &\leq \left(1-\frac{1}{n+1}\right)\|u_n\|_1,
\end{align*}
where $A_{\cdot j}$ is the $j$th column of a matrix $A$. It follows from \eqref{eqn:u_n}  that 
\begin{align*}
    \|u_{n+1}\|_1 
  &\leq \left(1-\frac{1}{n+1}\right)\|u_n\|_1  +
    Cn^{-1-\min(\epsilon,0.1,\gamma)} +Cn^{-1}\one(n\ \text{is "bad"}). 
\end{align*}
Iterating this bound from $n_0$ to $n\in[3^{k_0-1}n_0,3^{k_0}n_0]$ for
some $k_0\geq 1$ we obtain 
\begin{align*}
    \|u_{n}\|_1 
  &\leq \|u_{n_0}\|_1\prod_{m=n_0+1}^n\left(1-\frac{1}{m}\right) \\
    &+ C\sum_{d=n_0}^{n-1}  \left[ d^{-1-\min(\epsilon,0.1,\gamma)} +d^{-1}\one(d\ \text{is "bad"})
\right]\prod_{m=d+2}^n\left(1-\frac{1}{m}\right)
\\
  &\leq Cn^{-1}  + Cn^{-\min(\epsilon,0.1,\gamma)}
    +Cn^{-1}\sum_{k=1}^{k_0} (3^{k-1}n_0)^{0.2} 
    \leq   Cn^{-\min(\epsilon, 0.1,\gamma)}. 
\end{align*}

Furthermore, 
\begin{align*}
    \|v_n\|_1 &\leq n^{-1}\sum_{j=K(n)}^\infty \mathbb{E}[N_{j}(n)] + \sum_{j=K(n)}^\infty  p_j
\leq n^{-1}\mathbb{E}\bigl( e_n/K(n)\bigr) +C\sum_{j=K(n)}^\infty j^{-1-\gamma}\\
&\leq n^{-1}\bigl( (n+1)/(c_0n/2)\bigr) + CK(n)^{-\gamma}
\leq  Cn^{-\min(\gamma,1)}.
\end{align*}
Combining the two bounds results in 
$$
\sum_{k=0}^{\infty}\bigl|p_k-\mathbb{E}[N_k(n)]/(\pi
n)\bigr|=\|e_n\|_1/\pi = (\|u_n\|_1 +\|v_n\|_1)/\pi 
\leq Cn^{- \min(\epsilon, 0.1,\gamma) }.
$$
This, along with
\eqref{e:diff.2}, gives us 
\begin{align*}
    \sum_{k=0}^{\infty}\left|p_k-\mathbb{E}\big[
  N_k(n)/v_n\big]\right| &\leq
  \sum_{k=0}^{\infty}\big|p_k-\mathbb{E}[N_k(n)]/(\pi n)\big|
      + \sum_{k=0}^{\infty}\big|\mathbb{E}\big[N_k(n)/v_n\big]-\mathbb{E}[N_k(n)]/\pi n\big| \\
    &\leq Cn^{-\min(\epsilon, 0.1,\gamma)}+O(n^{-(1-\rho)})  \leq Cn^{-\kappa},
\end{align*}
with $\kappa = \min(\epsilon, 0.1,\gamma,(1-\rho))$. This proves
\eqref{e:total.var}, assuming existence of a sequence $(c_i)$ such
that the function $K$ in \eqref{e:Kn} satisfies \eqref{e:good.K}. We
now proceed to show this existence.

Consider first the   interval $[n_0,3n_0]$. Notice that
$$
\sum_{k\geq c_0n/2} k N_{k}(n)\leq e_n\leq n+1,
$$
implying that for large $n$, 
$$
\sum_{k\geq c_0n/2}\mathbb{E}[N_{k}(n)]\leq 2.1/c_0.
$$
If
$$
B_n = \bigl\{ k\geq c_0n/2:\, \mathbb{E}[N_{k}(n)]>n^{-0.1}\bigr\},
$$
then its cardinality satisfies $\# B_n\leq Cn^{0.1}$.  Let
\begin{align*}
A_n =& \bigl\{ c\in[5c_0/8,c_0]: \lfloor c'n_0\rfloor + \lfloor
c(n+1-n_0)\rfloor \in B_n\ \text{or} \ \lfloor c'n_0\rfloor + \lfloor
c(n+1-n_0)\rfloor +1 \in B_n\bigr\} \\
 =& \bigcup_{\substack{k+\lfloor c'n_0\rfloor \in B_n\\ \text{or
  }k+\lfloor c'n_0\rfloor +1\in B_n}}
  \left[\frac{k}{n+1-n_0},\frac{k+1}{n+1-n_0}\right) \cap [c_0/2,c_0].
\end{align*}
Its Lebesgue measure is at most
$$
\text{Leb} (A_n)\leq 2\# B_n/(n+1-n_0)\leq  \frac{Cn^{0.1}}{n+1-n_0}.
$$
Therefore,
\begin{align*}
&\int_{5c_0/8}^{c_0}\left( \sum_{n=n_0}^{3n_0}\mathbbm{1}_{\lfloor c'n_0\rfloor + \lfloor
c(n+1-n_0)\rfloor\in B_n \ \text{or} \ \lfloor c'n_0\rfloor + \lfloor
c(n+1-n_0)\rfloor +1 \in B_n    } \right) dc \\
                 =&\sum_{n=n_0}^{3n_0}|A_n| 
\leq    C\sum_{n=n_0}^{3n_0}\frac{n^{0.1}}{n+1-n_0}\ \leq
      Cn_0^{0.1}\log n_0.
\end{align*}
This implies that there must exist $c\in[5c_0/8,c_0]$ such that 
$$
\sum_{n=n_0}^{3n_0}\mathbbm{1}_{\lfloor c'n_0\rfloor + \lfloor
c(n+1-n_0)\rfloor\in B_n \ \text{or} \ \lfloor c'n_0\rfloor + \lfloor
c(n+1-n_0)\rfloor +1 \in B_n    }  \leq Cn_0^{0.2}.
$$
We use this $c$ as $c_1$ in \eqref{e:Kn} and, hence, we have
constructed the function $K(n)$ on the interval $[n_0,3n_0]$.  For the
next interval, we use its  left-end point
$3n_0$ as the starting point $n_0$ above and construct the function
$K(n)$ on that interval in the same way. The construction is
accomplished inductively. 
 
It remains to prove the statement \eqref{e:nonrand.conv} of Theorem
\ref{thm:convergence+concentration}. The main step is to establish the
following concentration result for the counts $N_k(n)$:   there is a
finite positive constant $C$ and $N$ such that such that for all
$n\geq N$, 
\begin{equation}\label{eqn:azuma}
     \mathbb{P}\bigl( |N_k(n)-\mathbb{E}[N_k(n)]| \geq t\bigr)  \leq 2e^{-C^{-1}t^2/(n (\log n)^{2})}
\end{equation}
for all $k$ and all $t> 0$.

Assuming, for a moment, that this is true, we derive
\eqref{e:nonrand.conv} as follows. By \eqref{e:total.var} and
\eqref{e:diff.2} we can
choose $N$ such that   for all $n\geq N$,
$$
    \sum_{k=0}^{\infty}\left|p_k-\mathbb{E}[N_k(n)/v_n]\right| \leq
    \epsilon/3 \ \ \text{and} \ 
    \sum_{k=0}^{\infty}\left|\mathbb{E}\big[N_k(n)/v_n\big]-\mathbb{E}[N_k(n)]/(\pi n)\right|\leq
    \epsilon/3. 
$$
Therefore, for such $n$, for all $k$,  by \eqref{eqn:azuma} and \eqref{approx:vertex_concentration}, 
\begin{align*}
   &\mathbb{P}\left( \left|N_k(n)/v_n-p_k\right| > \epsilon\right)
  \leq \mathbb{P}\left( \left| N_k(n)/v_n-\mathbb{E}[N_k(n)]/(\pi n)\right| > \epsilon/3\right)\\
    &\hskip 0.2in \leq \mathbb{P}\left( \left|N_k(n)/(\pi
      n)-\mathbb{E}[N_k(n)]/(\pi n)\right| > \epsilon/6\right)
      + \mathbb{P}\left( \left| N_k(n)/(\pi n)-N_k(n)/v_n\right| > \epsilon/6\right)\\
    &\hskip 0.2in \leq \mathbb{P}\left( \left|N_k(n)-\mathbb{E}[N_k(n)]\right| >
      \epsilon \pi n/6\right) 
      + \mathbb{P}\left( \left|v_n/(\pi n)-1\right| > \epsilon/6\right)\\
    &\hskip 0.2in \leq  Ce^{-C^{-1}\min(\epsilon,\epsilon^2)n/(\log
      n)^{2}}, 
\end{align*}
as required.

The inequality \eqref{eqn:azuma} is a concentration inequality and we
prove it by representing the random variable
$N_k(n)-\mathbb{E}[N_k(n)]$ as the sum of suitably bounded martingale
differences and then appealing to the Azuma inequality. Let $\calF(j)
= \sigma\bigl( G(1),\ldots, G(j)\bigr), \, j=0,1,2,\ldots$ be the
filtration generated by the growing graph. Then 
$$
M_{j}=\mathbb{E}[N_k(n)\mid \calF(j)], \qquad  j=0,1,\ldots, n
$$
is a Doob martingale with  $M_0=\mathbb{E}[N_k(n)]$ and
$M_n=N_k(n)$. We proceed to derive a uniform upper bound on the
differences $|M_j-M_{j-1}|, \, j=1,\ldots, n$. It is standard to check
that for each such $j$, with probability 1, 
 \begin{equation} \label{e:bound.osc}
|M_j-M_{j-1}|\leq
\sup_{H_1,H_2\in\mathcal{H}(G(j-1))}
\bigl|f_j(H_1)-f_j(H_2)\bigr|, 
\end{equation}
where $\calH(G(j-1))$ is the collection of all graphs that can arise
by performing a single system update step (either adding a vertex and
an edge or deleting an edge) starting with the graph
$G(j-1)$. Furthermore, for a fixed graph $H$,
\[
f_j(H):=\mathbb{E}\left[N_k(n)\mid G(j)=H\right].
\]
We refer the reader to \cite{vanderHofstad:2016} for more information
on \eqref{e:bound.osc}. The difference $f_j(H_1)-f_j(H_2)$ is the
difference of the expected values of $N_k(n)$ starting at time $j\leq
n$ from two difference graphs, $H_1$ and $H_2$, both of which have to
be in $\calH(G(j-1))$. To bound this difference we use coupling - we
couple the two sequences of graphs that arise starting at $H_1$ and
$H_2$. The coupling idea is also used in \cite{vanderHofstad:2016},
but there the coupling task is easier: indeed, $H_1$ and $H_2$
are similar because both 
arise by adding a node and an edge to the graph. In our case 
the two graphs may be less similar because one may arise by adding a
node and an edge to a given graph, wheres the second one may arise by
deleting an edge from the same graph. In particular, $H_1$ and $H_2$
may not have the same numbers of vertices and edges. 

We now spell out the details of our coupling procedure. 
Let $\mathcal{H}^{+}(G(j-1)),\, \mathcal{H}^{-}(G(j-1))\subseteq
\mathcal H(G(j-1))$ denote the sets of graphs obtained, respectively,
by adding a vertex and an edge to $G(j-1)$ and by deleting an edge in
$G(j-1)$. 

We first consider the case when $H_1,H_2\in \mathcal{H}^{+}(G(j-1))$.
In this case $H_1$ and $H_2$ have the same numbers of vertices and
edges. We denote by $v$ the new vertex added to $G(j-1)$ in $H_1$ and
$H_2$, thus identifying the vertices in these graphs. We also denote
by $u_i$ the vertex in $G(j-1)$ to which $v$ is attached in 
$H_i, \, i=1,2$. Then all existing vertices in $H_1$ and $H_2$ apart
from $u_1$ and $u_2$ have the same in-degrees in both graphs.  On a
new probability space $\bigl( \Omega^\prime, \calF^\prime,
\PP^\prime\bigr)$ we construct random sequences 
of graphs, $\bigl(G^{(i)}(j+1), \ldots, G^{(i)}(n)\bigr)$, 
arising according to the rules of our model starting at time
$j$ with the graph $H_i, \, i=1,2$.  
We couple the two sequences of
graphs by maintaining at all times equal numbers on vertices and edges
between the two sequences, which is the case at time $j$. Given that,
we synchronize ``additions'' and ``deletions'' in both
sequences. Starting at time $j+1$, if a new vertex is added in the
first sequence, and is connected to any vertex in $H_1$ other than
$u_1$ or $u_2$, this new vertex is then connected to the same vertex
in $H_2$ in the second sequence. If a new vertex is added in the
first sequence, and is connected to one of the vertices 
$u_1$ or $u_2$ in $H_1$, then this new vertex is also added to one of
these vertices (but not necessarily the same one)  in $H_2$ in the
second sequence. Similarly, if an edge is deleted in the first
sequence, whose endpoint is neither $u_1$ nor $u_2$, then the same
edge is deleted in the second sequence. If an edge whose endpoint is
either $u_1$ or $u_2$ is deleted in the first sequence, then also an
edge (not necessarily the same) whose endpoint is
either $u_1$ or $u_2$ is deleted in the second sequence.  In this way
the graphs obtained at time $j+1$ in both sequences have identical
sets of vertices, same numbers of edges, and all the vertices other
than $u_1$ or $u_2$ have identical in-degrees in both graphs. The
coupling procedure is then continued until time $n$, and ends with
graphs $G^{(1)}(n)$ and $G^{(2)}(n)$ still having identical
sets of vertices, same numbers of edges, and all the vertices other
than $u_1$ or $u_2$ have identical in-degrees in these 
graphs. Let $N_k^{(i)}(n)$ be the number of vertices of in-degree $k$ in
$G^{(i)}(n), \, i=1,2$. Then 
\begin{align} \label{e:bound.md}
\bigl|f_j(H_1)-f_j(H_2)\bigr| = \bigl| \EE^\prime
  \bigl(N_k^{(1)}(n)\bigr)-\EE^\prime \bigl(N_k^{(2)}(n)\bigr)\bigr|
  \leq \EE^\prime \bigl|N_k^{(1)}(n)-N_k^{(2)}(n)\bigr|
\leq 2.
  \end{align}
      
An identical coupling construction shows that \eqref{e:bound.md} still
holds if $H_1,H_2\in \mathcal{H}^{-}(G(j-1))$. We now consider the
most complicated case,  when $H_1\in \mathcal H^{+}(G(j-1))$ and
$H_2\in\mathcal H^{-}(G(j-1))$. Once again, on a
new probability space $\bigl( \Omega^\prime, \calF^\prime,
P^\prime\bigr)$ we construct random sequences 
of graphs, $\bigl(G^{(i)}(j+1), \ldots, G^{(i)}(n)\bigr)$, 
arising according to the rules of our model starting at time
$j$ with the graph $H_i, \, i=1,2$.  We use superscripts to
distinguish between quantities belonging to the two sequences. In the
obvious notation, the initial (at time $j$) numbers of edges and vertices in the two
sequences satisfy 
\[
e_j^{(1)} = e_j^{(2)} + 2,
\qquad
v_j^{(1)} = v_j^{(2)} + 1 .
\]
We identify the common vertices of the two graphs.

We construct a coupling of the two sequences of graphs as
follows. First of all, we synchronize  once again ``additions'' and ``deletions'' in both
sequences. This is possible as long as the current graph in the second
sequence has any edges. If this is not the case at any point, the
second sequence performs an ``addition'' as it must, while the first
sequence  chooses independently between ``addition'' and ``deletion'';
if it is ``deletion'', the deleted edge is also chosen independently. 
After this synchronization of
``additions'' and ``deletions'' resumes for as long as possible. This
guarantees that the two sequences satisfy  
\begin{equation} \label{e:compare}
e_i^{(2)} \le e_i^{(1)} \le e_i^{(2)}+2,
\qquad
 v_i^{(2)}  \leq v_i^{(1)} \leq  v_i^{(2)} + 1, \ i=j, \ldots, n.
\end{equation} 

Consider now  time $i=j+1,\ldots, n$ at  which a new vertex $v$ is
added to the existing graphs in both sequences.  We identify this vertex in the
two new graphs. The vertex $v$ attaches to an existing vertex $u\in
V_{i-1}^{(k)}$ or to itself with probabilities 
\[
p_i^{(k)}(u)
=
\frac{d_{i-1}^{(k)}(u)+\delta}{e_{i-1}^{(k)}+\delta v_{i-1}^{(k)}+\delta},
\qquad
p_i^{(k)}(v)
=
\frac{\delta}{e_{i-1}^{(k)}+\delta v_{i-1}^{(k)}+\delta}, \ k=1,2.
\]
The above choices are made in a coupled way, and we use  the maximal
coupling. That is, if $u\in V_{i-1}^{(1)}\cap V_{i-1}^{(2)}$, then
with probability
\begin{equation} \label{e:max.couple1}
\min\left(   \frac{d_{i-1}^{(1)}(u)+\delta}{e_{i-1}^{(1)}+\delta
    v_{i-1}^{(1)}+\delta}, \, \frac{d_{i-1}^{(2)}(u)+\delta}{e_{i-1}^{(2)}+\delta v_{i-1}^{(2)}+\delta}   \right) 
\end{equation}
the new vertex $v$ is attached to $u$ in both new graphs, and with
probability
\begin{equation} \label{e:max.couple2}
\min\left(   \frac{\delta}{e_{i-1}^{(1)}+\delta
    v_{i-1}^{(1)}+\delta},\, \frac{\delta}{e_{i-1}^{(2)}+\delta v_{i-1}^{(2)}+\delta}\right)
\end{equation}
the new vertex $v$ creates a self-loop in  both new graphs. The
remaining mass in the probability laws of the attachement choices in
the both sequences of graphs is distributed arbitrarily (possibly,
independently).

Similarly, suppose that  at time $i=j+1,\ldots, n$ an edge is deleted
from the existing graphs in both sequences. Recall that one can choose
an edge to delete by, first, choosing a vertex with probability
proportional to its in-degree, and then deleting a randomly chosen
edge entering this vertex; see Remark \ref{rk:edge.c}. The
probabilities to select an existing  vertex $u\in
V_{i-1}^{(k)}$ is $q_i^{(1)}(u)=d_i^{(k)}(u)/e_i^{(k)}$, $k=1,2$, and
we again couple the selection in the two sequences using the maximal
coupling:  if $u\in V_{i-1}^{(1)}\cap V_{i-1}^{(2)}$, then with
probability $\min\bigl(
d_i^{(1)}(u)/e_i^{(1)},d_i^{(2)}(u)/e_i^{(2)}\bigr)$ the vertex $u$ is
selected in both sequences of graphs and one of edges entering it in
each graph is randomly chosen and deleted (this choice is not
coupled). Once again, the remaining mass in the probability laws of
the selection choices in
the both sequences of graphs is distributed arbitrarily (possibly,
independently).

With the coupling described above, we follow the dynamics of the
number of vertices in the two sequences of graphs with possibly
unmatched in-degrees. Consider the following random sets defined on
the probability space $\bigl( \Omega^\prime, \calF^\prime,
\PP^\prime\bigr)$: 
\begin{align*}
&E_i = \bigl\{u\in V^{(1)}_i \cap V^{(2)}_i: d_i^{(1)}(u)= 
d_i^{(2)}(u)\bigr\}, \\
&D_i = \bigl\{u\in V^{(1)}_i \cap V^{(2)}_i: d_i^{(1)}(u)\neq
d_i^{(2)}(u)\bigr\} \cup \bigl( V^{(1)}_i \triangle V^{(2)}_i\bigr), 
 \, i= j-1, j, \ldots, n. 
\end{align*}
Clearly,  $D_{j-1}=\emptyset$. An important observation is that for
$u\in E_i$, the minima in \eqref{e:max.couple1} and
\eqref{e:max.couple2} are necessarily achieved on the first sequence
of graphs; see \eqref{e:compare}. The situation is the same in the
case of a ``deletion'' in the two sequences of graphs.

Notice that the cardinality of $V^{(1)}_i\setminus V^{(2)}_i$ is equal to 1 for all
$i=j,\ldots, n$, while the set $V^{(2)}_i\setminus V^{(1)}_i$ can increase with $i$ only
when $e_i^{(2)}=0$, and this can happen because of coupling at most
once: indeed, each time $V^{(2)}_i\setminus V^{(1)}_i$ increases, the difference in the
number edges between the graphs in the first and the second sequences
decreases by 2, and once that difference vanishes, addition of new
vertices is synchronized between the two sequences. Therefore, the
cardinality of $V^{(2)}_i\setminus V^{(1)}_i$ is bounded by 1 for all
$i=j,\ldots, n$. We define
$$
K_i=\big| \bigl\{u\in V^{(1)}_i \cap V^{(2)}_i: d_i^{(1)}(u)\neq
d_i^{(2)}(u)\bigr\} \big|,
$$
the cardinality of the main part of the ``mismatch'' set $D_i$. Let 
$\calF^\prime(i) = \sigma\bigl( G^{(d)}(j), \ldots,
G^{(d)}(i)$,  $d=1,2\bigr)$, $i=j,\ldots, n$.  On the event $\{
e_i^{(2)}>0\}\in \calF^\prime(i) $ we consider a random vector
$(X_+,Y_+)$ taking values in $V_i^{(1)}\times V_i^{(2)}$, and
representing the existing vertices in the graphs $G^{(1)}(i)$ and
$G^{(2)}(i)$, to which a new vertex added at time $i+1$ to both
graphs, connects under the coupling described above, if ``addition''
occurs at that time. Similarly, on the same event   we consider a random vector
$(X_-,Y_-)$, also taking values in $V_i^{(1)}\times V_i^{(2)}$, and
representing the existing vertices in the graphs $G^{(1)}(i)$ and
$G^{(2)}(i)$, chosen  under the coupling described above to have one
of their incoming edges deleted,  if ``deletion'' occurs at that
time. We already know that under the maximal coupling, if $X_+\in
E_i$, then $Y_+=X_+$, and if $X_-\in
E_i$, then $Y_-=X_-$. 
Therefore, on the event $\{
e_i^{(2)}>0\}$, 
\begin{equation*}\label{eqn:mismatches}
\begin{aligned}
   \mathbb{E}^\prime[K_{i+1}\mid \calF^\prime(i)]= K_i+\pi  
   \mathbb{P}^\prime ( X_+\notin E_i,Y_+\in E_i)  
 + (1-\pi) \mathbb{P}^\prime (X_-\notin E_i,Y_-\in E_i). 
\end{aligned}
\end{equation*}
 Note that by \eqref{e:compare}, 
\begin{align*}
    \mathbb{P}^\prime (X_+\notin E_i,Y_+\in E_i)&= \sum_{u\in E_i} \left(\frac{d^{(2)}_i(u)+\delta}{e_i^{(2)}+\delta v_i^{(2)}+\delta}-\frac{d^{(1)}_i(u)+\delta}{e_i^{(1)}+\delta v_i^{(1)}+\delta}\right)\\
        &\leq \sum_{u\in E_i} \left(\frac{d^{(2)}_i(u)+\delta}{e_i^{(2)}+\delta v_i^{(2)}+\delta}
        - \frac{ d^{(2)}_i(u)+\delta}{e_i^{(2)}+\delta
          v_i^{(2)} +2(1+\delta)}\right)\\
         &\leq \frac{2+\delta}{e_i^{(2)}+\delta v_i^{(2)}+2(1+\delta)}.
\end{align*}
Similarly,
\begin{align*}
  \mathbb{P}^\prime (X_-\notin E_i,Y_-\in E_i)\leq \frac{2}{e_i^{(2)}+2}.
\end{align*}
Therefore, 
\begin{align*}
    \mathbb{E}^\prime[K_{i+1}]&\leq \mathbb{E}^\prime[K_i]+\pi \mathbb{E}^\prime\left[\frac{2+\delta}{e_i^{(2)}+\delta v_i^{(2)}+2(1+\delta)}\right] + (1-\pi)\mathbb{E}^\prime\left[\frac{2}{e_i^{(2)}+2}\right].
 \end{align*}
If $e^{(2)}_{j}>0$, then for $i\geq j$ the sum $e_i^{(2)}+\delta
v_i^{(2)}$ can be viewed as belonging to a growing sequence of graphs
that starts at time zero with the graph $H_2$, and it stochastically
dominates the same sum in a growing sequence of graphs
that starts at time zero with $G(0)$ (two vertices and one
edge). Therefore, by \eqref{e:expect.frac},
\begin{equation} \label{e:stochD}
\mathbb{E}^\prime\left[\frac{2+\delta}{e_i^{(2)}+\delta
    v_i^{(2)}+ 2(1+\delta)}\right] \leq \frac{C}{i-(j-1)}
\end{equation}
for some absolute constant $C$.  If $e^{(2)}_{j}=0$, then the same
stochastic domination holds for all $i\geq j+1$, so \eqref{e:stochD}
still holds, perhaps with a larger absolute constant $C$. Similarly,
for $i\geq j$, 
$$
\mathbb{E}^\prime\left[\frac{2}{e_i^{(2)}+2}\right]\leq \frac{C}{i-(j-1)}
$$
for some absolute constant $C$, and so 
\begin{align*}
    \mathbb{E}^\prime[K_{i+1}]&\leq \mathbb{E}^\prime[K_i] +C\bigl(
                                i-(j-1)\bigr)^{-1}, \ i=j,\ldots, n, 
\end{align*} 
implying that 
\begin{align*}
\mathbb{E}^\prime[K_n]
&\le  C\sum_{i=1}^{n-j}1/i
\leq C\log (n-j+1) .
\end{align*}
Therefore, as in \eqref{e:bound.md}, 
\begin{align} \label{e:bound.md1}
\bigl|f_j(H_1)-f_j(H_2)\bigr| =& \bigl| \EE^\prime
   [N_k^{(1)}(n) ]-\EE^\prime  [N_k^{(2)}(n)]\bigr|
   \leq \EE^\prime \bigl[\bigl|N_k^{(1)}(n)-N_k^{(2)}(n)\bigr| \big]\\
\notag \leq& \EE^\prime  \big[|D_n|  \big]\leq C\log (n-j+1)+3. 
\end{align}

It follows by \eqref{e:bound.osc}, \eqref{e:bound.md} and
\eqref{e:bound.md1} that
$$
|M_j-M_{j-1}|\leq C\log (n-j+1)+3, \ j=1,\ldots, n.
$$
Therefore by Azuma's inequality for any $t>0$, 
\begin{align*}
   &\mathbb{P}^\prime( |N_k(n)-\mathbb{E}[N_k(n)]| \geq t) = \mathbb{P}(
   |M_n-M_0| \geq t) \\
\leq& 2\exp\left\{{-\frac{t^2}{C\sum_{j=1}^n \left(\log
  (n-j+1)+3\right)^{2}}}\right\}\leq  2e^{-C^{-1}t^2/(n (\log n)^{2})},
\end{align*}
proving \eqref{eqn:azuma} and completing the proof of Theorem
\ref{thm:convergence+concentration}.

\section{Tail asymptotics}  \label{sec:conclusion xc}
\setcounter{thm}{0}\setcounter{Corol}{0}\setcounter{lemma}{0}\setcounter{pron}{0}\setcounter{equation}{0}
\setcounter{remark}{0}\setcounter{exam}{0}\setcounter{property}{0}\setcounter{defin}{0}

In this section we prove Theorem~\ref{thm:tails}.  Recall the
constants $\alpha_0$ and $\alpha_2$  in \eqref{e:all.coeff} and
the numbers $A$ and $B$ in \eqref{e:AB}. We 
define a function $f$ on the interval $[0,1)$ by 
\begin{equation} \label{e:f1}
f(x) =  (1-x)^{-\delta} e^{1/(\alpha_0(1-x))}\alpha_0^{-1}
\int_x^1 (1-y)^{-2+\delta}    e^{-1/(\alpha_0(1-y))}  dy
\end{equation}
if $ \alpha_0=\alpha_2$,
\begin{equation} \label{e:f2}
  f(x)  = (\alpha_2-\alpha_0 x)^{A}  (1-x)^{B}  
    \int_x^1  (\alpha_2-\alpha_0 y)^{-(A+1)}  (1-y)^{-(B+1)}
    dy
\end{equation}
if $ \alpha_0<\alpha_2$, and 
\begin{equation} \label{e:f3}
f(x) = |\alpha_2-\alpha_0 x|^{A}  (1-x)^{B}
          \int_{\alpha_2/\alpha_0}^x (\alpha_0 y-\alpha_2)^{-1} |\alpha_0 y-\alpha_2|^{-A} (1-y)^{-(B+1)} dy
\end{equation}
if $ \alpha_0>\alpha_2$. The heart of the
proof of Theorem~\ref{thm:tails} is contained in the following proposition.  

\begin{proposition} \label{pr:heart}
The function $f$ has an extension to the complex plane that is
analytic in the open ball $\{ |z|<1\}$. The coefficients in the power
series
\begin{equation} \label{e:power.ser}
f(z) = \sum_{n=0}^{\infty} p_n z^n, \ |z|<1
\end{equation}
satisfy the recursion \eqref{eqn:rr}. Moreover, \eqref{e:pk.asymp}
holds. 
\end{proposition}
\begin{proof}
Beginning with function in \eqref{e:f1}, using the main branch of complex power
functions involved, the function of the complex variable $(1-z)^{-\delta}
e^{1/(\alpha_0(1-z))}$ is analytic in the unit radius open ball
around the origin, and so is the function $(1-z)^{-2+\delta}
e^{-1/(\alpha_0(1-z))}$. Extending the integral in \eqref{e:f1} as an
integral over a rectifiable path (say, over the straight line) from
$x$ to $1/2$ (say) is then analytic function of $x$ in the open ball. Since the real integral
$$
\int_{1/2}^1 (1-y)^{-2+\delta}    e^{-1/(\alpha_0(1-y))}  dy<\infty,
$$
this gives us the required analytic extension. The analytic extension
of the function $f$ in \eqref{e:f2} to the unit radius open ball
around the origin is constructed in the same way.

In the case of the function $f$ in \eqref{e:f3} we start by noticing
that, changing the variable of integration we can rewrite it in the
form
\begin{equation} \label{e:f3.alt}
f(x) = (1-x)^B\alpha_0^{-1}\int_0^1 t^{-(A+1)} \bigl(
1-\alpha_2/\alpha_0 -(x-\alpha_2/\alpha_0)t\bigr)^{-(B+1)}dt
\end{equation}
(defined by continuity as $B/|A|$ for $x=\alpha_2/\alpha_0$). For
every fixed $0<t<1$ the function of a complex-valued $x$ given by the main
branch of $\bigl(
1-\alpha_2/\alpha_0 -(x-\alpha_2/\alpha_0)t\bigr)^{-(B+1)}$ is
analytic inside the unit radius open ball
around the origin, so its integral (over the real $t$) satisfies
inside that ball the Cauchy-Riemann equations, hence is
analytic. Extending once again the function $(1-x)^B$ to an analytic
function over the ball, gives us the required analytic extension of the
function $f$ in \eqref{e:f3}.

Consider the restriction of the power series \eqref{e:power.ser} to
the real line; it converges at least on the interval $(-1,1)$. Since
$f$ is real-valued on that interval, the coefficients $(p_n)$ are
real. In all three cases, differentiating the function $f$ (of the
real variable) in a small interval around the origin and elementary
simplification shows that, in that small interval it satsfies the
ordinary differential equation
\begin{equation} \label{e:ode}
f^\prime(x)\bigl( \alpha_2+\alpha_1x+\alpha_2x^2\bigr) +f(x)\bigl(
\beta_1+\beta_0x\bigr) = -1.
\end{equation}
The left-hand side of this relation is a function with power expansion
$\sum q_nx^n$ around the origin with coefficients satisfying
$$
q_n=\alpha_2(n+1)p_{n+1}+(\beta_1+\alpha_1n)p_n
+(\beta_0+\alpha_o(n-1))p_{n-1}, \ \ n=0,1,2,\ldots,
$$
with $p_{-1}=0$. Since we must have $q_0=-1$ and $q_n=0$ for $n\geq
1$, this shows that the sequence $(p_n)$ satisfies the recursion
\eqref{eqn:rr}. 

It remains to prove that the sequence $(p_n)$ also satisfies
\eqref{e:pk.asymp}. Here we will use different arguments for three
cases \eqref{e:f1}, \eqref{e:f2} and \eqref{e:f3}.

We start with the function $f$ in \eqref{e:f1}, so
$\alpha_0=\alpha_2$. Changing the variable of integration to
$u=(y-x)(1-x)^{-1}(1-y)^{-1}/\alpha_0$ gives us, after elementary
calculations, 
$$
    f(x) =  \int_0^\infty  (1+\alpha_0(1-x)u)^{-\delta}\cdot  e^{-u}  du.
    $$
This identifies $f$ as the moment generating function of the negative
binomial distributions, all with shape parameter $\delta$ and
probabilities for success $p_u=1/(1+\alpha_0u)$, over values of $u>0$
chosen from the standard exponential distribution. Therefore, the
coefficients in the power series for $f$ are given by 
\begin{align*}
   p_k &= \int_0^\infty   e^{-t}\cdot \frac{\Gamma(k+\delta)}{k!\Gamma(\delta)}(1-p_t)^k p_t^\delta \, dt \\
   &= \int_0^\infty   e^{-t}\cdot \frac{\Gamma(k+\delta)}{k!\Gamma(\delta)}(\alpha_0t)^k (1+\alpha_0 t)^{-(k+\delta)} \, dt \\
   &= \frac{1}{\Gamma(\delta)}\int_0^\infty
     s^{\delta-1}e^{-s}\frac{(\alpha_0s)^k}{(1+\alpha_0s)^{k+1}}\, ds \\
&=\frac{k^{\delta/2}}{\Gamma(\delta)}\int_0^\infty
     t^{\delta-1}e^{-k^{1/2}t}\frac{(\alpha_0k^{1/2}t)^k}{(1+\alpha_0k^{1/2}t)^{k+1}}\,
                                                                           dt,      \ k=0,1,2,\ldots, 
\end{align*}
after easy manipulations. Let $0<\vep<\alpha_0^{-1/2}$ be a small
number. Note that, as $k\to\infty$, 
\begin{align*}
&\int_0^\vep
  t^{\delta-1}e^{-k^{1/2}t}\frac{(\alpha_0k^{1/2}t)^k}{(1+\alpha_0k^{1/2}t)^{k+1}}dt
\leq \int_0^\vep t^{\delta-1} \left( 1-\frac{1}{1+\alpha_0k^{1/2}t}\right)^kdt\\
  \leq& \delta^{-1}\vep^\delta \left( 1-\frac{1}{1+\alpha_0k^{1/2}\vep}\right)^k
        =O\bigl( e^{-k^{1/2}/(\alpha_0\vep)}\bigr). 
\end{align*}

Next we write
\begin{align*}
  &\int_\vep^\infty     t^{\delta-1}e^{-k^{1/2}t}
    \frac{(\alpha_0k^{1/2}t)^k}{(1+\alpha_0k^{1/2}t)^{k+1}}dt\\
=&\int_\vep^\infty  e^{1/(2\alpha_0^2t^2)}   \frac{t^{\delta-1}}{(1+\alpha_0k^{1/2}t)}
\exp\left\{ -k^{1/2}\bigl( t+1/(\alpha_0t)\bigr)\right\}I_k(t)\, dt,
 \end{align*}                
 where
 $$
 I_k(t)=\exp\left\{ k^{1/2}\left[
     \frac{1}{\alpha_0t}-k^{1/2}\log\left(1+\frac{1}{\alpha_0k^{1/2}t}\right)\right]
   - \frac{1}{2\alpha_0^2t^2}\right\}.
$$
Clearly, $I_k(t) =1+o(1)$ as $k\to\infty$ uniformly in
$t\in[\vep,\infty)$. Since the function $\phi(t)= t+1/(\alpha_0t)$ has
a unique minimum at the point $t_*=\alpha_0^{-1/2}$, we see that, as
$k\to\infty$, 
\begin{align*}
& \int_\vep^\infty     t^{\delta-1}e^{-k^{1/2}t}
    \frac{(\alpha_0k^{1/2}t)^k}{(1+\alpha_0k^{1/2}t)^{k+1}}dt\\
\sim&\int_\vep^\infty  e^{1/(2\alpha_0^2t^2)}   \frac{t^{\delta-1}}{(1 +\alpha_0k^{1/2}t)}
\exp\left\{ -k^{1/2}\bigl( t+1/(\alpha_0t)\bigr)\right\} dt\\
\sim& \, e^{1/(2\alpha_0^2t_*^2)} t_*^{\delta-1}\frac{1}{\alpha_0t_*}k^{-1/2}
\int_\vep^\infty  
\exp\left\{ -k^{1/2} \phi(t)\right\} dt.
\end{align*}
The behaviour of the final integral can be obtained using the
standard Laplace method for the asymptotic expansions of integrals;
see e.g. pp. 183-185 in
\cite{bleistein:handelsman:1975}. Specifically, as $k\to\infty$, 
$$
\int_\vep^\infty  
\exp\left\{ -k^{1/2} \phi(t)\right\} dt \sim
\sqrt{\frac{2\pi}{k^{1/2}\phi''(t_*)}}e^{-k^{1/2}\phi(t_*)}. 
$$
Putting everything together and choosing $\vep<1/(2\alpha_0^{1/2})$
proves \eqref{e:pk.asymp} in the case $\alpha_0=\alpha_2$. 

We now switch to proving \eqref{e:pk.asymp} in the case
$\alpha_0<\alpha_2$, so the function $f$ is given by \eqref{e:f2}.
Unlike the case $\alpha_0=\alpha_2$ where we could identify the
function $f$ as the moment generating function of a recognizable
probability law, in the present case this is easy to achieve only for
certain choices of the parameters. We instead appeal to the transfer
theorem contained in Theorem \ref{t:transfer} in the appendix.  We
start by showing that the function $f$ in \eqref{e:f2} can be extended
to a $\Delta$-analytic function in a $\Delta$-domain at
$\zeta=\alpha_2/\alpha_0$. As in the appendix, we use the principal
values for all complex logarithms and power functions. Changing the
variable of integration by $y=x+(1-x)t$, we can write
\begin{equation} \label{e:ext.2}
f(x)  = (\alpha_2-\alpha_0 x)^{A} \int_0^1 \bigl( \alpha_2-\alpha_0t
-x\alpha_0(1-t)\bigr)^{-(A+1)} (1-t)^{-(B+1)}\, dt, \ 0\leq x<1.
\end{equation}

Fix any $0<\phi<\pi_*/2$. The function $z\mapsto (\alpha_2-\alpha_0
z)^{A}$ is analytic in the domain
$\bigl\{ z\not=\alpha_2/\alpha_0, \,
|\text{arg}(z-\alpha_2/\alpha_0)|>\phi\bigr\}$, and the same is true
for the function $z\mapsto \bigl( \alpha_2-\alpha_0t
-z\alpha_0(1-t)\bigr)^{-(A+1)}$ for any $0<t<1$.  As above, this gives
the desired analytic extension of $f$ to the domain
$$
\bigl\{ z \mid |z|<R,\ z\neq \alpha_2/\alpha_0,\
|\arg(z-\alpha_2/\alpha_0)|>\phi \}
$$
with any $R>\alpha_2/\alpha_0$, as desired. 

Suppose first that $A<0$. Identifying $f$ in \eqref{e:ext.2} with its
extension, we see that the integral is continuous over the domain at
the singularity $\alpha_2/\alpha_0$ and, at the singularity, is equal to
$$
(\alpha_2-\alpha_0)^{-(A+1)}\int_0^1 t^{-(A+1)} (1-t)^{-(B+1)}\, dt
= (\alpha_2-\alpha_0)^{-(A+1)} B_*(-A,-B),
$$
where $B_*$ is the standard Beta function.   Then 
$$
f(z)  = (\alpha_2-\alpha_0 z)^{A} (\alpha_2-\alpha_0)^{-(A+1)}
B_*(-A,-B) + f_1(z),
$$
where $f_1$ is a $\Delta$-analytic function such that
$f_1(z) = o\bigl( (\alpha_2-\alpha_0 z)^{A}\bigr)$ as
$z\to\alpha_2/\alpha_0$ over $\Delta$. 
Now Theorem \ref{t:transfer}  proves \eqref{e:pk.asymp} in the case
$\alpha_0<\alpha_2$ and $A<0$. 

Still in the case $\alpha_0<\alpha_2$, we now consider the situation
$A\geq 0$ and $A$ not an integer. Let $m$ be an integer such that
$m< A<m+1$.  In our $\Delta$-domain at
$\alpha_2/\alpha_0$ we represent $f$ as a sum of $\Delta$-analytic
functions 
\begin{equation} \label{e:splitf.gh}
f(z)  = g(z)+h(z),
\end{equation}
where
$$
g(z)=(\alpha_2-\alpha_0 z)^{A} \sum_{j=0}^m {-(B+1)\choose j} (-1)^j \int_0^1 \bigl( \alpha_2-\alpha_0t
-z\alpha_0(1-t)\bigr)^{-(A+1)} t^j \, dt 
$$
and
$$
h(z) = (\alpha_2-\alpha_0z)^{A} \int_0^1 \bigl( \alpha_2-\alpha_0t
-z\alpha_0(1-t)\bigr)^{-(A+1)} R_m(t)\, dt
$$
with 
$$
R_m(t)= (1-t)^{-(B+1)}-\sum_{j=0}^m {-(B+1)\choose j} (-1)^jt^j. 
$$
For $z$ in the intersection of a small neighborhood of the
singularity with the domain, elementary integration shows that 
\begin{align*}
g(z)=&  -(\alpha_2-\alpha_0z)^A\frac{(\alpha_2-\alpha_0)^{-A}}{A} \sum_{j=0}^m {-(B+1)\choose
       j} (-1)^j\sum_{k=0}^j {j\choose k}\bigg/{A-1\choose k} \frac{(\alpha_2-\alpha_0)^k}{\alpha_0^{k+1}(z-1)^{k+1}}
\\
+& A^{-1}\sum_{j=0}^m  \frac{(\alpha_2-\alpha_0z)^j}{\alpha_0^{j+1}(z-1)^{j+1}} 
(-1)^j{-(B+1)\choose  j}
\biggr/{A-1\choose j}.
\end{align*}
For $z$ as above and every $\ell=1,2,\ldots$,
$$
\frac{1}{z-1} = \sum_{d=1}^\ell \frac{(\alpha_2/\alpha_0-z)^{d-1}}{(\alpha_2/\alpha_0-1)^{d}}
+
\frac{(\alpha_2/\alpha_0-z)^\ell}{(z-1)(\alpha_2/\alpha_0-1)^{\ell}}. 
$$
Using this with $\ell=1$, we see that for any $k\geq 0$,
$$
(z-1)^{-(k+1)} = \bigl( \alpha_2/\alpha_0-1\bigr) ^{-(k+1)}
+\sum_{d=1}^{k+1} a_{d,k} \frac{(\alpha_2/\alpha_0-z)^d}{(z-1)^d}
$$
for some constants $(a_{d,k})$. Furthermore,  with $\ell=m +1$ we can
write 
$$
(z-1)^{-(j+1)} =  \sum_{d'=1}^{j+1}\sum_{d=(m+1)d'}^{(j+1)(m+1)} b_{d, d',k} 
\frac{(\alpha_2/\alpha_0-z)^{d}}{(z-1)^{ d'}} + P_j\bigl( \alpha_2/\alpha_0-z\bigr)
$$
for some constants $(b_{d,d^\prime,k})$ and a polynomial of a finite order
$P_j$. Therefore, we can rewrite the function $g$ in the form
\begin{align} \label{e:g.final}
g(z)=& - (\alpha_2-\alpha_0z)^A\frac{(\alpha_2-\alpha_0)^{-(A+1)}}{A}  \sum_{j=0}^m {-(B+1)\choose
  j} (-1)^j\sum_{k=0}^j {j\choose k}\bigg/{A-1\choose k}\\
\notag +& P_*\bigl( \alpha_2/\alpha_0-z\bigr) + g_1(z),
\end{align}
where $P_*$ is a polynomial of a finite order and $g_1$ is a $\Delta$-analytic
function such that $g_1(z) = o\bigl( (\alpha_2-\alpha_0z)^A\bigr)$ as
$z\to \alpha_2/\alpha_0$ in $\Delta$. 

Next, the integral in the definition of the function $h$ s continuous over the domain at
the singularity $\alpha_2/\alpha_0$ and, at the singularity, is equal to
$$
(\alpha_2-\alpha_0)^{-(A+1)}\int_0^1 t^{-(A+1)} R_m(t)\, dt. 
$$
Therefore, we can write
\begin{align} \label{e:h.final}
h(z) =  (\alpha_2-\alpha_0z)^{A} (\alpha_2-\alpha_0)^{-(A+1)}\int_0^1
  t^{-(A+1)} R_m(t)\, dt + h_1(z), 
\end{align} 
where $h_1$ is a $\Delta$-analytic
function such that $h_1(z) = o\bigl( (\alpha_2-\alpha_0z)^A\bigr)$ as
$z\to \alpha_2/\alpha_0$ in $\Delta$.

We conclude by \eqref{e:g.final} and \eqref{e:h.final} that
\begin{equation} \label{e:final.f.nonint}
f(z) = (\alpha_2-\alpha_0z)^{A} d_{\pi,\rho}(\alpha_2-\alpha_0)^{-(A+1)}+ P_*\bigl( \alpha_2/\alpha_0-z\bigr) + f_1(z),
\end{equation} 
where $f_1$ is a $\Delta$-analytic
function such that $f_1(z) = o\bigl( (\alpha_2-\alpha_0z)^A\bigr)$ as
$z\to \alpha_2/\alpha_0$ in $\Delta$ and
\begin{align} \label{e:constant}
d_{\pi,\rho} = -A^{-1}\sum_{j=0}^m {-(B+1)\choose
  j} (-1)^j\sum_{k=0}^j {j\choose k}\bigg/{A-1\choose k}+
\int_0^1  t^{-(A+1)} R_m(t)\, dt, 
\end{align} 
and so by Theorem \ref{t:transfer},
$$
p_k\sim
(\alpha_2-\alpha_0)^{-(A+1)}\alpha_2^A
d_{\pi,\rho}(\alpha_0/\alpha_2)^k k^{-(A+1)}/\Gamma(-A).
$$
We now show that
\begin{equation} \label{e:const.nonint.proof}
d_{\pi,\rho} = \Gamma(-B)\Gamma(-A)/\Gamma(\delta),   
\end{equation}
which will imply \eqref{e:pk.asymp}. Indeed, letting $r_k = {j\choose
  k}/{A-k\choose k}$ for $k=0,\ldots,j$, we see that
$r_{k+1}/r_k=(j-k)/(A-1-k)$, so
$$
(A-k)r_k-(A-k-1)r_{k+1}=(A-j)r_k.
$$
Summing up this relation over $k=0,\ldots, j-1$ gives us $\sum_{k=0}^j
r_k = A/(A-j)$. If we denote $c_j(B)=(-1)^j {-(B+1)\choose j}$, then
\begin{align*}
d_{\pi,\rho} = \sum_{j=0}^m \frac{c_j(B)}{j-A}+\int_0^1
  t^{-(A+1)}\sum_{j=m+1}^\infty c_j(B)t^j\, dt
  =  \sum_{j=0}^\infty \frac{c_j(B)}{j-A}. 
\end{align*}
We view temporarily the right-hand side above as a function of
independent variables $A$ and $B$, which we call $d(A,B)$. We keep
$B<0$ fixed. Since
$c_0=1$ and $c_{j+1}(B)/c_j(B)=(B+1+j)/(j+1)$, we can write
\begin{align*}
d(A+1,B)=& -\frac{1}{A+1} + \sum_{j=0}^\infty \frac{c_{j+1}(B)}{j-A}
= -\frac{1}{A+1} + \sum_{j=0}^\infty \frac{c_{j}(B)}{j-A}\frac{B+j+1}{j+1}\\
=&-\frac{1}{A+1} +\frac{A+1+B}{A+1}d(A,B) -\frac{B}{A+1}\int_0^1
   c_j(B)t^j\, dt\\
=&-\frac{1}{A+1} +\frac{A+1+B}{A+1}d(A,B) -\frac{B}{A+1}\int_0^1
   (1-t)^{-(B+1)}\, dt \\
=& \frac{A+1+B}{A+1}d(A,B). 
\end{align*}
Therefore,
$$
d(A,B)= \prod_{i=0}^m \frac{A+B-j}{A-j}d(A-m-1,B).
$$
Since $A-m<1$,we can compute
\begin{align*}
d(A-m-1,B) =& \sum_{j=0}^\infty \frac{c_j(B)}{j-A+m+1} 
= \int_0^1 \sum_{j=0}^\infty c_j(B)t^{j-A+m}\, dt\\
=& \int_0^1 t^{-(A-m)}(1-t)^{-(B+1)}\, dt =
   \frac{\Gamma(-A+m+1)\Gamma(-B)}{\Gamma(-A-B+m+1)}. 
\end{align*}
Recalling again that $A+B=-\delta$, we conclude that
$$
d_{\pi,\rho} = d(A,B) = \prod_{i=0}^m \frac{A+B-j}{A-j} 
\frac{\Gamma(-A+m+1)\Gamma(-B)}{\Gamma(-A-B+m+1)} =
\Gamma(-B)\Gamma(-A)/\Gamma(\delta),
$$
as required. This completes the proof in the case
$\alpha_0<\alpha_2$ and $A>0$ is not an integer.

Suppose now that $\alpha_0<\alpha_2$ and $A$ is a nonnegative
integer.  Returning to \eqref{e:ext.2}, let $\sum_{n=0}^\infty q_n x^n$ be the
power series expansion of the integral in a neighborhood of the
origin. Then for $n\geq A$,
\begin{equation} \label{e:pn.qn}
p_n = \sum_{j=0}^A (-1)^j {A \choose j} \alpha_0^j\alpha_2^{A-j}
q_{n-j}.
\end{equation}
In order to evaluate the coefficients $(q_m)$, we rewrite the integral
as
$$
(\alpha_2-\alpha_0)^{-(A+1)} \int_0^1 \left(
  \frac{\alpha_2-\alpha_0}{\alpha_2-\alpha_0t-x\alpha_0(1-t)}\right)^{A+1}(1-t)^{-(B+1)}dt.
$$
Identifying the resulting integral as a mixture of the generating
functions of negative binomial distributions with $A+1$ successes and
probability for success $(\alpha_2-\alpha_0)/(\alpha_2-\alpha_0t)$, we
see that
\begin{align*}
q_n =& (\alpha_2-\alpha_0)^{-(A+1)} \int_0^1 (1-t)^{-(B+1)} \left(
  \frac{\alpha_2-\alpha_0}{\alpha_2-\alpha_0t}\right)^{A+1} \left(
  \frac{\alpha_0(1-t)}{\alpha_2-\alpha_0t}\right)^n {n+A \choose A}\, dt\\
=& {n+A \choose A} \alpha_0^n \int_0^1
   (1-t)^{n-(B+1)}(\alpha_2-\alpha_0t)^{-(n+A+1)}dt \\
=& {n+A \choose A} \alpha_2^B (\alpha_2-\alpha)^{\delta-1} \bigl(
   \alpha_0/\alpha_2\bigr)^n \int_0^1 u^{n-(B+1)}\bigl( 1-(\alpha_0/\alpha_2)u\bigr)^{-\delta}du\\
=& {n+A \choose A} \alpha_2^B (\alpha_2-\alpha)^{\delta-1} \bigl(
   \alpha_0/\alpha_2\bigr)^n \int_0^1 u^{n-(B+1)} \sum_{l=0}^\infty
   {-\delta\choose l} \bigl( -\alpha_0/\alpha_2\bigr)^l u^l du\\
=& {n+A \choose A} \alpha_2^B (\alpha_2-\alpha)^{\delta-1} \bigl(
   \alpha_0/\alpha_2\bigr)^n \sum_{l=0}^\infty
   {-\delta\choose l}  \bigl( -\alpha_0/\alpha_2\bigr)^l \frac{1}{n-B+l},
\end{align*}
where the third equality follows by the change of variable
$u=\alpha_2(1-t)/(\alpha_2-\alpha_0t)$.   
Substituting this expression into  \eqref{e:pn.qn} we obtain  
\begin{align*} 
  p_n = & \alpha_2^{-\delta} (\alpha_2-\alpha)^{\delta-1}
          (\alpha_0/\alpha_2)^n \sum_{l=0}^\infty
   {-\delta\choose l}  \bigl( -\alpha_0/\alpha_2\bigr)^l
\sum_{j=0}^A (-1)^j {A 
          \choose j} \frac{{n+A-j \choose A}}{n-j-B+l}.
\end{align*}
Notice that, for each fixed $l$, ${x+A\choose A}$ is a polynomial of
degree $A$ in a real variable $x$. Subtracting its value at $x_0=B-l$
makes $x_0$ a root, so there is a polynomial
$P_{A-1}(x)$ of degree $A-1$ such that
$$
{x+A\choose A} = (x-B+l) P_{A-1}(x) + {B-l+A\choose A}.
$$
Recalling that
\begin{equation} \label{e:eq2}
\sum_{j=0}^A (-1)^j {A 
  \choose j} j^d=0 \ \ \text{for} \ d=0,\ldots, A-1,
\end{equation}
we see that
$$
\sum_{j=0}^A (-1)^j {A 
  \choose j} P_{A-1}(n-j)=0. 
$$
Therefore,
\begin{align*}
p_n =&  \alpha_2^{-\delta} (\alpha_2-\alpha)^{\delta-1}
          (\alpha_0/\alpha_2)^n  
 \sum_{l=0}^\infty
   {-\delta\choose l}  \bigl( -\alpha_0/\alpha_2\bigr)^l{B-l+A\choose
     A} \\
&\hskip 2in\sum_{j=0}^A (-1)^j {A 
          \choose j} \frac{1}{n-j-B+l}\\
=&\alpha_2^{-\delta} (\alpha_2-\alpha)^{\delta-1}
          (\alpha_0/\alpha_2)^n  
 \sum_{l=0}^\infty
   {-\delta\choose l}  \bigl( -\alpha_0/\alpha_2\bigr)^l{B-l+A\choose
     A} \prod_{d=0}^A\frac{1}{n-B-l-d} (-1)^A A!, 
\end{align*}        
where the second equality can be checked using, once again,
\eqref{e:eq2}.  By the dominated convergence theorem, as $n\to\infty$, 
\begin{align*}
n^{A+1} &\sum_{l=0}^\infty
   {-\delta\choose l}  \bigl( -\alpha_0/\alpha_2\bigr)^l{B-l+A\choose
     A} \prod_{d=0}^A\frac{1}{n-B-l-d} \\
\to &\sum_{l=0}^\infty
   {-\delta\choose l}  \bigl( -\alpha_0/\alpha_2\bigr)^l{B-l+A\choose
     A} \\
= &{-\delta\choose A} \sum_{l=0}^\infty
   {-\delta-A\choose l}  \bigl( -\alpha_0/\alpha_2\bigr)^l
=   {-\delta \choose A} \bigl( 1-\alpha_0/\alpha_2\bigr)^{-\delta-A},
\end{align*}
and simple algebra proves \eqref{e:pk.asymp} in this case as well.

It remains to prove \eqref{e:pk.asymp} in the case
$\alpha_0>\alpha_2$. Now we take the function $f$ as given by
\eqref{e:f3.alt} and let $g(x) = f\bigl( (\alpha_2/\alpha_0)x\bigr),
\, 0\leq x<1$. It is elementary to check that $g$ can be written in
the form 
$$
g(x) = (\alpha_0-\alpha_2 x)^{B} \int_0^1 \bigl( \alpha_0-\alpha_2t
-x\alpha_2(1-t)\bigr)^{-(B+1)} (1-t)^{-(A+1)}\, dt, \ 0\leq x<1, 
$$
which is exactly the same function as that given by \eqref{e:ext.2}
with the roles of $\alpha_0$ and $\alpha_2$ interchanged, and the
roles of $A$ and $B$ interchanged. Since we already know the
asymptotic behaviour of the coefficients in the power series expansion
of the function in \eqref{e:ext.2} near the origin, the coefficients
$(\tilde p_k)$ in the power series expansion of $g$ near the origin
satisfy
$$
\tilde p_k\sim
(\alpha_0-\alpha_2)^{-(B+1)}\alpha_0^B\Gamma(-A)/\Gamma(\delta)
(\alpha_2/\alpha_0)^k k^{-(B+1)} \ \ \text{as} \ k\to\infty.
$$
Since $\tilde p_k = (\alpha_2/\alpha_0)^kp_k$ for every $k$ by the
definition of $g$, the statement of the proposition in the case
$\alpha_0>\alpha_2$ follows. 

\end{proof}

It is now easy  to prove Theorem \ref{thm:tails}. We have proved that
the limiting probabilities $(p_k)$ in
Theorem~\ref{thm:convergence+concentration} coincide with the
coefficients in the series expansions of the functions in
Proposition~\ref{pr:heart}, and the statement of Theorem
\ref{thm:tails}  follows from that proposition.

\begin{remark}
  {\rm 
The origin of the functions in \eqref{e:f1}, \eqref{e:f2} and
\eqref{e:f3} may initially appear somewhat mysterious. In fact,
starting with the recursion \eqref{eqn:rr}, one can first derive the
differential equation \eqref{e:ode} the moment generating function of
the potential limiting distribution must satisfy, and then solve the
equation to obtain the functions in \eqref{e:f1}, \eqref{e:f2} and
\eqref{e:f3}. Since the origin of these functions is irrelevant for
the argument, we have avoided describing the above process in detail.
}
\end{remark}

\appendix
\section{Approximations}\label{sec:approx}
Here we perform the necessary work to bound the deviations of the
random numbers of edges and vertices in $G(n)$, i.e. of $e_n$ and
$v_n$, from their linear approximations.

Let $((X_n,Y_n))$ be a random walk taking values in $\bbz^2$ that starts
with $(X_0,Y_0)=(1,2)$ and takes a step $(1,1)$ with probability
$\pi$, and a step $(-1,0)$ with probability $1-\pi$. 
We couple the number of edges process $(e_n)$ with the simple random
walk $(X_n)$. The coupling is obtained by running the 2-dimensional random walk
$((X_n,Y_n))$, and then constructing a realization of $(e_n)$ as
follows. The process $(e_n)$ takes the same steps as $(X_n)$ until the
first time $X_n$ becomes negative. At this moment $(e_n)$ moves to the 
point 1 and then takes the same steps as the process $(X_n)$ does after the
first return of the latter to the point 1. If the process $(X_n)$
becomes negative once again, the story repeats. This way the value of
$e_n$ is equal to the value of the random walk at a time point that
may be delayed until after time $n$. However, the delay does not
exceed the overall time $D$ the random walk spends in $\bbz_{\leq 0}$.
Since the random walk takes steps of size 1, we conclude that with
this coupling, $|e_n-X_n|\leq D$ for all $n$. 

It is well known (and easy to check) that $D$ has an exponentially
fast decaying tail: for some $0<\tau<1$, $\PP(D\geq m)\leq \tau^m$ for
all $m$, so under this coupling, 
$$
\mathbb{P}(|e_n-X_n|\geq m) \leq  \tau^m \ \text{for all}\ n,\, m.
$$
Applying Hoeffding's inequality to the deviations of the random walk
$(X_n)$ from its mean we see that   for any $1/2 < \rho < 1$ there is
$C_\rho\in (0,\infty)$ such that
\begin{equation}\label{approx:edge}
    \mathbb{P}\bigl(\big|e_n-(2\pi-1)n\big|>n^{\rho}\bigr) \leq
    e^{-n^{2\rho-1}/C_\rho} \ \text{for all} \ n. 
\end{equation}

Using the same 2-dimensional random walk $((X_n,Y_n)$ we couple the
number of vertices process $(v_n)$ with the random walk $(Y_n)$ in the
same manner: The process $(v_n)$ takes the same steps as $(Y_n)$ until the
first time $X_n$ becomes negative. At this moment $(v_n)$ takes a step
equal to 1 and then takes the same steps as the process $(Y_n)$ does
after the 
first return of the process $(X_n)$ to the point 1, etc.  With the
same $D$ as before, we then have $|v_n-Y_n|\leq D$ for all $n$, and we
conclude in the same manner that for any $1/2 < \rho < 1$ there is
$C_\rho\in (0,\infty)$ such that
\begin{equation}\label{approx:vertex}
    \mathbb{P}(|v_n-\pi n|> n^\rho)\leq  e^{- n^{2\rho-1}/C_\rho} \
    \text{for all} \ n. 
\end{equation}
A similar argument with $\rho=1$  says that there is $C_1\in
(0,\infty)$ such that for all $a>0$ 
\begin{equation}\label{approx:vertex_concentration}
    \mathbb{P}(|v_n-\pi n|>an)     \leq C_1\exp\bigl\{ - C_1 \min(a,a^2) \, n\bigr\} \ \text{for all} \ n.
\end{equation}

Here are some simple conclusions from the above bounds. It follows
from either \eqref{approx:edge} or  from \eqref{approx:vertex} that
\begin{align} \label{e:expect.frac}
    \mathbb{E}\left[1/(e_n+ 1)\right] = O(n^{-1}), \   \ \mathbb{E}\left[1/(e_n+\delta v_n+\delta)\right] = O(n^{-1}).
\end{align}
Next, since $\sum_k kN_k(n) =e_n$, we have by \eqref{approx:edge}, for 
any $1/2<\rho<1$, 
\begin{align} \label{e:diff.1}
    & \sum_{k=0}^{\infty} \left|\mathbb{E}\left[\frac{kN_k(n)}{e_n}\mathbbm{1}_{\{ e_n > 0 \}}\right]-\frac{\mathbb{E}\left[kN_k(n) \right]}{(2\pi-1)n}\right| 
    \leq   \sum_{k=0}^{\infty} \mathbb{E}\left[kN_k(n)\left|\frac{\mathbbm{1}_{\{ e_n > 0 \}}}{e_n}-\frac{1}{(2\pi-1)n}\right|\right]\\
\notag     & 
    \leq  \mathbb{E}\left[\left|1-\frac{e_n}{(2\pi-1)n}\right| \, \one\bigl(
      \big|e_n-(2\pi-1)n\big|\leq n^{\rho}\bigr)\right] \\
&+\mathbb{E}\left[\left|1-\frac{e_n}{(2\pi-1)n}\right| \, \one\bigl(
      \big|e_n-(2\pi-1)n\big|> n^{\rho}\bigr)\right]
       = O(n^{-(1-\rho)}). \notag 
\end{align}
Similarly
\begin{align} \label{e:diff.2}
\sum_{k=0}^{\infty}\left|\mathbb{E}\Big[\frac{N_k(n)}{v_n}\Big]-\frac{\mathbb{E}[N_k(n)]}{\pi n}\right| = O(n^{-(1-\rho)})
\end{align}
and
\begin{align} \label{e:diff.3}
 & \sum_{k=0}^{\infty}  \left|\mathbb{E}\left[\mathbbm{1}_{\{ e_n > 0 \}}\frac{(k+\delta)N_k(n)}{e_n+\delta
   v_n+\delta}\right]-\frac{(k+\delta)\mathbb{E}\left[N_k(n)
   \right]}{(2\pi-1)n+\delta \cdot \pi n}\right|  = O(n^{-(1-\rho)}).
\end{align}

\section{A transfer theorem} \label{sec:transfer}
\setcounter{thm}{0}\setcounter{Corol}{0}\setcounter{lemma}{0}\setcounter{pron}{0}\setcounter{equation}{0}
\setcounter{remark}{0}\setcounter{exam}{0}\setcounter{property}{0}\setcounter{defin}{0}

Transfer theorems relate the behaviour of certain complex-valued
functions analytic in a domain near singularities at the boundary of
the domain to the asymptotic behaviour of the coefficients in their
series expansion.  We follow \cite{flajolet:sedgewick:2009}. All
complex logarithms and power functions in this section are understood as their
principal branches. For a
function $f$ analytic in a neighborhood of the origin we denote by
$[z^n]f(z)$ the coefficient $a_n$ in its series expansion at the
origin, $f(z)=\sum_n a_nz^n$.   
\begin{defin}
    Given two numbers $\phi, R$ with $R>1$ and $0<\phi<\pi/2$  let 
\[
\Delta(\phi,R)
=
\{ z \mid |z|<R,\ z\neq 1,\ |\arg(z-1)|>\phi \}.
\]
A domain is a $\Delta$-domain at $1$ if it is $\Delta(\phi,R)$ for some $R$ and $\phi$.
For a complex number $\zeta\neq 0$, a $\Delta$-domain at $\zeta$ is the image under the mapping
$z\mapsto \zeta z$ of a $\Delta$-domain at $1$.
A function is $\Delta$-analytic if it is analytic in some $\Delta$-domain.
\end{defin}

The next statement is the key transfer theorem we will use; it is a
special case of Theorems VI.1, VI.3 and VI.4 in \cite{flajolet:sedgewick:2009}.

\begin{thm} \label{t:transfer}
Let $f(z)$ be a $\Delta$-analytic function
for a $\Delta$-domain at 
$\zeta>0$ such that
$$
f(z) = (1-z/\zeta)^{-\theta}+ o\bigl( (1-z/\zeta)^{-\alpha}\bigr)
$$
as $z\to\zeta$ in $\Delta$, 
for $\theta\in\bbc\setminus\bbz_{\leq 0}$ and $\alpha\in\bbr$. Then
$$
[z^n]f(z) =   \zeta^{-n}n^{\theta-1} /\Gamma(\theta) + o\bigl(
|\zeta|^{-n}n^{\alpha-1} 
\bigr)
$$
as $n\to\infty$.
\end{thm}

\bibliographystyle{/Users/gs18/Documents/GenaFiles/texfiles/authyear}

\end{document}